\documentclass[table,svgnames]{article} 

\usepackage{amsmath,amsfonts,bm}

\def\eqref#1{equation~\ref{#1}}

\def\1{\bm{1}}

\def\eps{{\varepsilon}}

\DeclareMathAlphabet{\mathsfit}{\encodingdefault}{\sfdefault}{m}{sl}
\SetMathAlphabet{\mathsfit}{bold}{\encodingdefault}{\sfdefault}{bx}{n}

\DeclareMathOperator*{\argmin}{arg\,min}

\usepackage[utf8]{inputenc} 
\usepackage[T1]{fontenc}    

\usepackage{url}            
\usepackage{booktabs,tabularx,ragged2e}       
\usepackage{nicefrac}       
\usepackage{microtype}      

\usepackage{graphicx}
\usepackage{mathrsfs}
\usepackage{bm}
\usepackage{framed}
\usepackage{tcolorbox}
\usepackage{subcaption}
\usepackage{mathtools}
\usepackage{cancel}

\usepackage[accepted]{icml2024}
\usepackage{hyperref}
\usepackage{svg}
\usepackage{multirow}

\definecolor{darkmidnightblue}{rgb}{0.0, 0.2, 0.4}
\definecolor{darkpowderblue}{rgb}{0.0, 0.2, 0.6}
\definecolor{dukeblue}{rgb}{0.0, 0.0, 0.61}

\hypersetup{
    colorlinks = true,
    citecolor=blue,
    urlcolor=blue,
    breaklinks=true,
    linkcolor = dukeblue,
    linkbordercolor = {white},
}

\usepackage[capitalize,noabbrev]{cleveref}
\usepackage{autonum}
\usepackage{float}
\usepackage{tikz}
\usetikzlibrary{positioning}
\tikzset{
    my/.style={
        draw=green,thick,fill=white,circle,minimum width=.5cm
    },
    >=latex
}
\usetikzlibrary{calc}

\def\LIPERM1{\hyperlink{LIPERM}{\textup{\sf{LIPERM}}}}

\newcommand{\AssA}[2]{\hyperlink{AssA}{\textup{\textsf{A}}}$({#1},{#2})$}

\def\hat{\widehat}
\def\rmd{{\rm d}}

\def\Id{\operatorname{\sf Id}}

\renewcommand{\leq}{\leqslant}
\renewcommand{\geq}{\geqslant}
\renewcommand{\le}{\leqslant}
\renewcommand{\ge}{\geqslant}

\def\bU{\boldsymbol{U}}
\def\bu{\boldsymbol{u}}
\def\bx{\boldsymbol{x}}
\def\bX{\boldsymbol{X}}
\def\bZ{\boldsymbol{Z}}
\def\wass{{\sf W}}

\usepackage[textwidth=2.0cm, textsize=small]{todonotes} 

\usepackage{amsmath, amsfonts, amssymb, amsthm}
\usepackage{xspace}

\newtheorem{proposition}{Proposition}
\newtheorem{theorem}{Theorem}
\newtheorem*{theorem*}{Theorem}
\newtheorem{lemma}{Lemma}
\newtheorem{corollary}{Corollary}

\theoremstyle{definition}

\theoremstyle{remark}
\newtheorem{remark}{Remark}

\newtheorem*{consequence*}{Consequence}

\icmltitlerunning{Generative Modeling with Maximum Deviation from the Empirical Distribution}

\begin{document}

\twocolumn[
\icmltitle{Statistically Optimal Generative Modeling with Maximum Deviation from the Empirical Distribution}



\icmlsetsymbol{equal}{*}

\begin{icmlauthorlist}
\icmlauthor{Elen Vardanyan}{YSU,equal}
\icmlauthor{Sona Hunanyan}{YSU,equal}
\icmlauthor{Tigran Galstyan}{YSU,YNN,equal}
\icmlauthor{Arshak Minasyan}{CREST}
\icmlauthor{Arnak Dalalyan}{CREST}
\end{icmlauthorlist}

\icmlaffiliation{YNN}{YerevaNN, Armenia}
\icmlaffiliation{YSU}{Department of Mathematics, Yerevan State University (YSU), Armenia}
\icmlaffiliation{CREST}{CREST, GENES, Institut Polytechnique de Paris, France}

\icmlcorrespondingauthor{Elen Vardanyan}{evardanyan@aua.am}
\icmlcorrespondingauthor{Arnak Dalalyan}{arnak.dalalyan@ensae.fr}

\icmlkeywords{generative models, Wasserstein GAN, minimax rate, sample complexity}

\vskip 0.3in
]


%

\newcommand{\fix}{\marginpar{FIX}}
\newcommand{\new}{\marginpar{NEW}}

\printAffiliationsAndNotice{\icmlEqualContribution}

\begin{abstract}
This paper explores the problem of generative modeling, aiming to simulate diverse examples from an unknown distribution based on observed examples. While recent studies have focused on quantifying the statistical precision of popular algorithms, there is a lack of mathematical evaluation regarding the non-replication of observed examples and the creativity of the generative model. We present theoretical insights into this aspect, demonstrating that the Wasserstein GAN, constrained to left-invertible push-forward maps, generates distributions that not only avoid replication but also significantly deviate from the empirical distribution. Importantly, we show that left-invertibility achieves this without compromising the statistical optimality of the resulting generator. Our most important contribution provides a finite-sample lower bound on the Wasserstein-1 distance between the generative distribution and the empirical one. We also establish a finite-sample upper bound on the
distance between the generative distribution and the true data-generating one. Both bounds are explicit and show the impact of
key parameters such as sample size, dimensions of the ambient and latent spaces, noise level, and smoothness measured by the Lipschitz constant.
\end{abstract}

\addtocontents{toc}{\protect\setcounter{tocdepth}{0}}

\section{Introduction}

Generative modeling is a widely-used machine learning technique that has found applications in various scientific and industrial domains, including health \citep{yan2018deeplesion,nie2017medical}, climate \citep{gagne2020machine}, finance \citep{wiese2020quant}, energy \citep{fekri2019generating}, physics \citep{paganini2018calogan}, chemistry \citep{maziarka2020mol}, and biology \citep{repecka2021expanding}. The primary goal of generative models is to simulate new examples by learning from training data, while ensuring diversity and avoiding the replication of examples from the training set.

Assessing the performance of a generative model can be done qualitatively by evaluating the realism of the generated examples, which we refer to as accuracy. However, accuracy should be balanced with another crucial property: the diversity of generated examples and their difference from the training examples. This property, referred to as the generator's creativity \cite{li2024good}, is essential to avoid overfitting (producing examples that are slight modifications of those in the training set). Qualitative evaluation of diversity is challenging due to the large size of training sets, making it impossible to retain all the examples they contain. Nonetheless, diversity is as important as accuracy, particularly in applications where generative models aim to enrich datasets in cases where data acquisition is expensive or infeasible. Generating examples that closely resemble the observed data diminishes the utility of such algorithms.

The success of deep neural nets in generative modeling has attracted significant attention from the machine learning community. The number of proposed methods in recent years, following the influential work by \cite{goodfellow2014generative}, is extensive, making it impractical to cite all of them here\footnote{For a comprehensive list, refer to \url{https://github.com/hindupuravinash/the-gan-zoo}}. While many of these methods have been empirically validated and justified using heuristics, a more comprehensive mathematical quantification of their strengths and limitations is often lacking.

The importance of diversity in generated examples is well-acknowledged but presents practical challenges. Numerous papers have empirically studied the issue of limited diversity in learned distributions, proposing some solutions \cite{DumoulinBPLAMC17,Arora0LMZ17,AroraRZ18,VEEGAN}. However, the majority of studies on diversity have primarily targeted the mitigation of mode collapse. This phenomenon occurs when certain examples in the training/testing set are  overlooked by the learned distribution; as a result, some test set examples are markedly dissimilar to those generated by the
learned distribution.

In this paper, we explore another aspect of diversity within the learned distribution. We aim to ensure that the examples it generates are not mere copies of the samples in the training set but rather novel instances. To achieve this, we seek to understand to what extent the learned distribution can deviate from the empirical distribution of observations while still closely adhering to the true underlying law.

From a theoretical standpoint, endeavors to address ``mode collapse'' primarily involve proposing methodological enhancements that aim to achieve a better precision, where the precision is understood as the closeness of the learned distribution (in Wasserstein or KL, for instance) to the true underlying law. The idea is that if the learned distribution closely aligns with the true distribution in these metrics, it is less likely to miss important modes of the true distribution. However, none of these approaches quantifies the dissimilarity between the learned distribution and the empirical distribution of the training set. This is a critical consideration, as the training process typically involves fitting the empirical distribution with a parametric class. The ultimate goal, however, is to generate examples that differ from those in the training set.

\textbf{Contributions} The key insight of this paper is the following: if the learned distribution is defined as the push-forward of the uniform distribution by a smooth map, then ensuring the push-forward map has a smooth left inverse prevents overfitting to the empirical distribution and promotes creativity. Moreover, when left invertibility is imposed on an already statistically optimal generator, it seems to retain its optimality. This claim, though speculative, is supported by our study of Wasserstein GANs (WGAN) \citep{pmlr-v70-arjovsky17a, GulrajaniAADC17}. We introduce LIPERM (Left-Inverse Penalized Empirical Risk Minimizer), a penalized version of WGAN that favors left invertibility of the push-forward map. Our main result establishes a lower bound on the Wasserstein-1 distance between the learned distribution and its empirical counterpart. We then establish
an upper bound on the precision measured by an integral probability metrics. It takes the form of a finite-sample
risk bound describing the behavior of the learned generator
as a function of the sample size $n$ and the dimension $d$
of the latent space. Importantly, these bounds are
independent of the ambient dimension $D$ and are
rate-optimal, as confirmed by lower bounds in \citep{schreuder_brunel_dalalyan_2021, tang2022minimax}.

When the latent dimension $d\geq 2$, we prove that LIPERM's separation from any distribution concentrated on the training sample is at least of order $n^{-1/d}$, while its precision  is established to be at most of order $n^{-1/d}$. Notably, $n^{-1/d}$ is the rate of approximating the true data-generating distribution with its empirical counterpart \citep{dudley_1969, Boissard, weed2022estimation}.
Consequently, a generative model separated from the empirical distribution by a distance larger than $n^{-1/d}$ would
have suboptimal precision. Thus, LIPERM, being sufficiently distant from distributions replicating observed examples, ensures diverse and novel example generation.

In \Cref{fig:1}, we illustrate the explored framework on the task of generating points on a 2D spiral using a 1D latent space. In the left panel,
the map $g: [0,1] \to \mathbb{R}^2$ corresponds to a generator with high diversity and a small left-inverse penalty (LIP);  the generated examples exhibit a wide range of variations and deviations from the training distribution. In the second panel, the map $g$ is inferior to the first one in terms of diversity, resulting in less varied generated examples, and it has a higher LIP. A closer examination reveals that for certain points on the spiral that are close to each other, their preimages under $g$ are located far apart. This means that the left-inverse of $g$ has a large Lipschitz norm.
Finally, in the rightmost panel, $g$ generates only a few
examples, producing a restricted set of outputs. It has a
LIP equal to $+\infty$. This illustration highlights the
connection between diversity and the left-inverse penalty.

\begin{figure*}
    \centering
    \includegraphics[width = 0.3\textwidth]{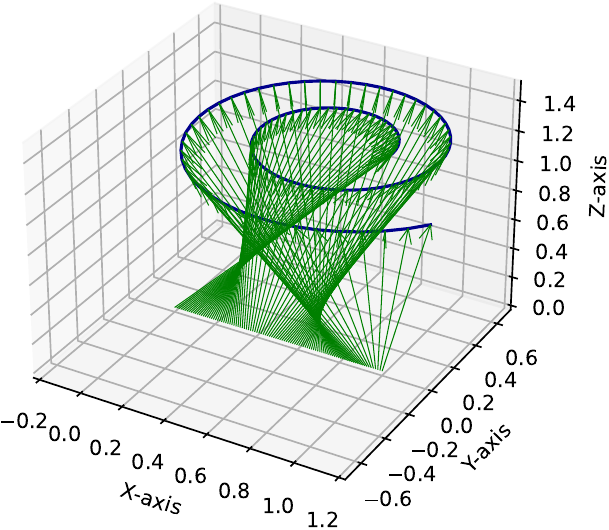}
    \includegraphics[width = 0.3\textwidth]{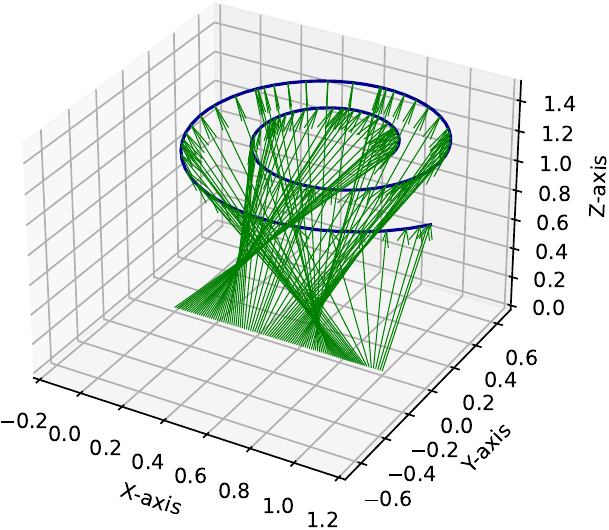}
    \includegraphics[width = 0.3\textwidth]{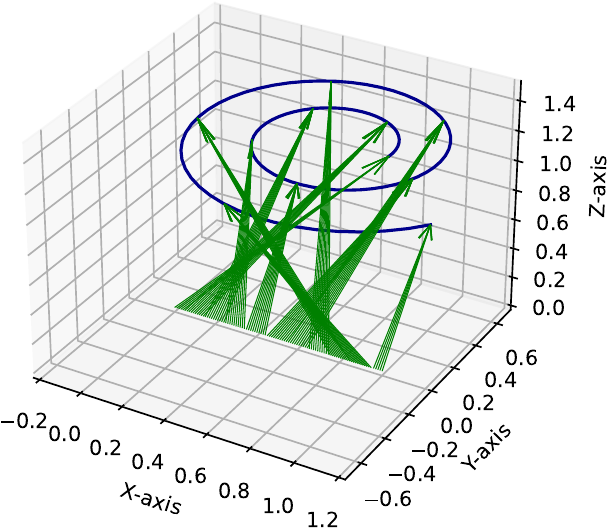}
    \caption{Illustration of the framework of this paper: generating points on a 2D spiral using a 1D latent space. The green arrows represent the mapping $g:[0,1]\to \mathbb R^2$. Each arrow indicates how points from the latent space are mapped to positions in the 2D spiral.}
    \label{fig:1}
\end{figure*}

\textbf{Prior work}
In recent years, a notable surge in papers has focused on mathematical aspects of generative models. Some treat generative models as tools
for distribution and density estimation, establishing their optimality
\citep{Liang, belomestny2023rates, BiauST21, Biau, uppal2019nonparametric,kwon2023minimax,chae2023likelihood}.
The convergence rate of adversarial generative models---under the
manifold assumption---independent of the ambient space dimension,
is highlighted in \citep{huang_etal_2022, schreuder_brunel_dalalyan_2021,
tang2022minimax, stéphanovitch2023wasserstein} using integral
probability metrics. The intricate relationship between minimax
optimality and distribution learning is explored in \citep{chen2022minimax}.

The analysis of diffusion-based generative models is presented in \citep{Bortoli1,debortoli2022riemannian}, while the investigation of autoencoders and their relation to the Langevin process is carried
out in \cite{Block}.
A regularization scheme for training GANs, based on adding a penalty
on the weighted gradient-norm of the discriminator, is introduced in
\citep{roth_gan_regulariz}; see also \citep{petzka2018on}. The
theoretical characterization of the mode-seeking behavior of general
$f$-divergences and Wasserstein distances, along with a guarantee
for mixtures, is provided in \citep{pmlr-v206-ting-li23a}. In the
context of generator invertibility and mode collapse in GANs, \citep{bai_ma_risteski_2019} suggest that invertible generators
might effectively alleviate mode collapse. \cite{xi_reddy_2023}
propose a theoretical framework for analyzing the indeterminacies
of latent variable models.

The issue of generative models memorizing training data and potentially copying from it has been emphasized in early works such as \citep{nagarajan2018theoretical} and \citep{GulrajaniRM19}. Recent observations have further revealed instances of duplicated training examples, particularly in diffusion models for image generation and large language models \citep{somepalli2023diffusion, somepalli2023understanding, daras2023ambient, carlini2021extracting, carlini2023extracting, jagielski2023measuring}. Advancing theoretical understanding in this realm could contribute to the development of algorithms mitigating the risk of replication.

\textbf{Notation}
For every integer $d>1$, we denote by $\mathcal{U}_d$
the uniform distribution on $[0,1]^d$. The norm $\|\bx\|$ of an element $\bx$ from an Euclidean space
is always the Euclidean norm. We denote by $\mathbb E[\bX]$ the expectation of a random variable. If necessary, we write $\mathbb E_P[\bX]$ to stress
that the expectation is considered under the condition that $\bX$ is drawn from $P$.
For a random vector $\bX$ and a real number $q\geq 1$,  we use the notation $\|\bX\|_{\mathbb L_q} = \mathbb E^{1/q}[\|\bX\|^q]$.
For two subsets $A$ and $B$ of some Euclidean spaces,
and a positive number $L$, we say that a function
$f:A\to B$ is $L$-Lipschitz-continuous, if
$\|f(\bx) - f(\bx')\|\leq L\|\bx - \bx'\|$ for every
$\bx,\bx'\in A$. The set of all the $L$-Lipschitz
continuous functions from $A$ to $B$ is denoted by
$\textup{Lip}_L(A\to B)$. The Dirac mass at a point $\bx$ is denoted by $\delta_{\bx}$. The notation
$P_{n,Z}$ is often used to design the empirical
distribution $(1/n)\sum_{i=1}^n\delta_{\bZ_i}$
($Z$ can be replaced by other letters).
We set ${\sf V}_d = \pi^{d/2}/\Gamma(1+d/2)$
to be the volume of the unit ball. Notation
$\Id_d$ stands for the identity mapping $\Id_d(\bx)
=\bx$ on $\mathbb R^d$, or any subset of it. For two
sets of functions $\mathcal G$ and $\mathcal H$, we
set $\mathcal G_{\mathcal H}$ the subset of
$\mathcal G$ the elements of which admit a left
inverse function in $\mathcal H$.

\section{Left-Inverse-Penalized Empirical Risk}\label{sec:2}

Let $P^*$ be the distribution of training data $\bX_1,\ldots,\bX_n$ in the $D$-dimensional Euclidean space $\mathbb R^D$ equipped with the Borel $\sigma$-algebra $\mathscr B(\mathbb R^D)$. While $D$ is typically large, we assume that the examples $\bX_i$ may originate from latent variables in a lower-dimensional space of dimension $d$.
This is closely related to the so-called ``manifold
assumption'' \citep{fefferman2016,NarayananM10}. To capture
this, several popular algorithms such as GAN or WGAN,
choose an integer $d>0$ much smaller than $D$ and seek to learn a generative distribution $\hat P_n$ as a smooth
transformation of the uniform distribution in $[0,1]^d$.
This is the setting considered in this paper: the trained
generative distribution is chosen of the form $P_g = g\sharp
\mathcal\,\mathcal U_d$, where $g:[0,1]^d\to\mathbb R^D$ is
called the push-forward map and $P_g$ defined by
$P_g(A) = P(g^{-1}(A))$, $\forall A\in\mathscr B(\mathbb
R^D)$, is the push-forward distribution.

In this paper, we study learned distributions
obtained by minimizing the penalized empirical risk using a suitable penalty. Let $\mathcal{H}$ and $\mathcal G$ be two functional classes such that
\begin{align}
    \mathcal H&\subset \text{Lip}_{L_{\mathcal H}}([0,1]^D\to[0,1]^d),\qquad \\ \mathcal{G}&\subset \text{Lip}_{L_{\mathcal G}}([0,1]^d\to[0,1]^D).
\end{align}
Let $\sf d$ be a distance on the space
of probability distributions. For $q\geqslant 1$, we define the
left inverse penalty
\begin{align}
    \text{pen}_{\mathcal H}(g) &= \min_{h\in \mathcal{H}}
    \int_{[0,1]^d} \| h \circ g(\bu) - \bu\|^q \rmd\bu\\
    &= \min_{h\in \mathcal{H}} \|h\circ g -
    \Id_d\|_{\mathbb L_q}^q. \label{LIP}\tag{LIP}
\end{align}
We then define the penalized empirical risk
\begin{align}\label{penL}
    \hat L^{{\sf d}, \mathcal{H}}_n(g) = {\sf d}(g\sharp\, \mathcal{U}_d, P_{n,X}) + \lambda\, \text{pen}_{\mathcal H} (g)
\end{align}
for a tuning parameter $\lambda>0$.  The
learned generator is the push-forward distribution $\widehat{g}_n\sharp\mathcal U_d$, with $\hat g_n = \hat g_n(\lambda,\mathcal{G},\sf d,\mathcal{H})$ being a solution to the \hypertarget{LIPERM}{minimization problem}
\begin{align}\label{LIPERM}
\tag{LIPERM}
    \widehat{g}_{n} \in \argmin_{g \in \mathcal{G}}
     \hat L^{{\sf d}, \mathcal{H}}_n(g).
\end{align}
The choices of the distance $\sf d$ and of  $\lambda$ are important. For $\sf d$ we will mainly use the Wasserstein-1
distance $\wass_1$ or a more general integral probability
metric (IPM) defined by
\begin{align}\label{IPM}\tag{IPM}
    {\sf d}_{\mathcal F}(P,Q) =
    \sup_{f\in\mathcal F} \big|
    \mathbb E_P[f(\bX)] - \mathbb E_Q[f(\bX)]\big|,
\end{align}
where $\mathcal F$ is a set of test functions
$\mathbb R^D\to\mathbb R$, and  $P,Q$  are two
probability measures on $([0,1]^D,\mathscr{B}([0,1]^D))$.
The Wasserstein-$1$ distance corresponds to an IPM
with $\mathcal{F}$ being the set of all $1$-Lipschitz-continuous functions.

As for $\lambda$, the case $\lambda = 0$ corresponds to
the standard setting of minimal distance estimator
including GAN, WGAN, and their variants.
The other extreme case $\lambda =+\infty$ corresponds
to minimizing the training error ${\sf d}(g\sharp\, \mathcal{U}_d, P_{n,X})$ under the constraint
$g\in\mathcal G_{\mathcal H}$. Although this
constrained estimator has some attractive properties,
its computation might be more challenging than that
of penalized one.

Note that the sets $\mathcal{G}$ and $\mathcal{H}$ are entirely determined by the user's choice, as are $\sf d$
and $\lambda$. We refer to $\hat g_n$ defined by (\ref{LIPERM}) as the {left-inverse-penalized empirical risk minimizer}. Although most of our results will refer to the
general form (\ref{penL}) of the loss function, it might be
useful for the reader to keep in mind the central example
corresponding to the case  $\sf d = \wass_1$ and
$q=2$ leading to the min-max problem
\begin{align}\label{penL2}
    \min_{g\in\mathcal G}\max_{h\in\mathcal H}\Big\{
    \wass_1(g\sharp\,
    \mathcal{U}_d, P_{n,X}) + \lambda\,
    \int\|h\circ g(\bu) - \bu\|^2\,\bu\Big\}.
\end{align}
The rationale behind considering this learning procedure is as follows: When training a generative model, the objective is to produce a distribution that (a) is easy to sample from,  (b) is close to the distribution of the training data, and (c)  avoids replicating the examples in the training dataset. The cost function in (\ref{LIPERM}) consists of two terms, each contributing to one of the last two desired properties. If the term ${\sf d}(g\sharp\, \mathcal{U}_d, P_{n,X})$ is small, the generator $g\sharp\mathcal U_d$ closely approximates the empirical distribution. Additionally, if $\text{pen}_{\mathcal H}(g)$ is small,  $g$ is nearly invertible with a smooth inverse. Intuitively, when $g$ possesses a smooth inverse, it disperses the unit hypercube $[0,1]^d$ across a large region in $[0,1]^D$ rather than concentrating around a small neighborhood of the observations from the
training set. In particular, the following simple fact holds
true.

\begin{lemma}\label{lem:1a}
    Any distribution $P_g = g\sharp\,\mathcal U_d$
    defined by a push-forward map $g\in \mathcal
    G_{\mathcal H}$ has no atom. In particular, it
    satisfies $P_g(\{\bX_1,\ldots,\bX_n\}) = 0$.
\end{lemma}

This lemma implies, in particular, that if the learned distribution is defined by \ref{LIPERM} with
$\lambda=+\infty$, then the
probability of generating an example that was present in
the training set is equal to zero.
\begin{proof}[Proof of Lemma~\ref{lem:1a}]
    Clearly, for any $\boldsymbol{x}\in\mathbb R^D$,
    $\hat P_n(\{\bx\}) = \textup{Leb}_d\big(
    g^{-1}(\bx) \big)$. Since $g$ has a left inverse,
    $g^{-1}(\bx)$ contains at most $1$ point, resulting
    in a Lebesgue measure of zero.
\end{proof}

Motivated by practical convenience, instead of an exact solution of \ref{LIPERM}, we will consider an
$\varepsilon$-approximate solution satisfying
\begin{equation}
\hat L^{{\sf d}, \mathcal{H}}_n(\widehat{g}_{n, \eps})
\leq \min_{g \in \mathcal{G}}
     \hat L^{{\sf d}, \mathcal{H}}_n(g) + \eps,
     \label{LIPERMe} \tag{LIPERMe}
\end{equation}
where $\epsilon>0$ is a small number. It is clear that
\begin{align}
    L^{\sf d,\mathcal{H}}(\widehat{g}_{n}) \leq L^{\sf d\mathcal{H}}(\widehat{g}_{n, \eps}) \leq L^{\sf d,\mathcal{H}}(\widehat{g}_{n}) + \eps.
\end{align}
The subsequent sections aim to mathematically characterize the properties of (\ref{LIPERMe}).

The left-inverse penalty in our work is similar to methods in deep learning, such as variational autoencoders \citep{kingma_welling_2013} and Cycle-GAN \citep{zhu_park_isola_efros_2017}. Variational autoencoders  learn a compact latent representation and generate new examples by sampling from the learned latent space, using the penalty
$\|g\circ h - \Id_D\|_{\mathbb L_2}$ instead of (\ref{LIP}). Cycle-GAN  focuses on style transfer and image-to-image translation, enforcing cyclic consistency similar to our LIP.
Specifically, Cycle-GAN aims to satisfy the conditions \( F \circ G(x) \approx x \) and \( G \circ F(y) \approx y \) using two neural networks that perform style transfer in opposite directions.

The resemblance between our left-inverse penalty and methods yielding favorable empirical results suggests that minimizing the penalized empirical risk, though challenging, is feasible. We leverage this resemblance in the accompanying implementation, leading to experimental results reported in \Cref{sec:6new}.

\section{Main Result: Deviation from the Empirical
Distribution}\label{sec:4}

If the class $\mathcal G$ is rich, it is likely to contain
a function $\hat g$ that overfits the training data:
the distance $\wass_1(\hat g\sharp\,
\mathcal U_d,P_{n,X})$ might be very small or even zero.
This type of overfitting has been observed in practice,
as highlighted in \citep{somepalli2023diffusion, somepalli2023understanding, daras2023ambient, carlini2021extracting, carlini2023extracting, jagielski2023measuring}.
This behavior is undesirable for most generative modeling applications, such as image or music generation. Simply resampling examples from the training set is not the intended outcome.

The main finding in this paper is that overfitting to the empirical distribution can be mitigated or substantially restrained by imposing constraints on admissible generators, specifically requiring them to have a smooth left inverse. This holds true, particularly for the learned distribution defined by (\ref{LIPERM}) with $\lambda = \infty$, referred to as the hard constraint case. In this setting, we establish that the learned distribution is significantly distant from the empirical distribution. Furthermore, we extend this finding to generators that are nearly left invertible, such as the generator defined by (\ref{LIPERMe}) with a $\lambda>0$.

To demonstrate that the learned generator,
$\hat g_n$, does not replicate the examples from the
training set, we examine the distance between the
probability measure induced by the generator, $\hat g_n\sharp\,\mathcal{U}_d$, and any distribution
$Q$ satisfying $Q(\{\bX_1,\ldots,\bX_n\}) = 1$. A
larger distance indicates a greater dissimilarity,
which is desirable.

\subsection{Warm-up: The Case of Hard Constraint   ($\lambda=\infty$) }

We first consider the generator $\hat g_n$ obtained by imposing the hard constraint $h\circ g = \text{Id}_d$ on feasible solutions. This means that $\hat g_n$ minimizes
the distance, ${\sf d}$, between $g\sharp\, \mathcal{U}_d$
and the empirical distribution $P_{n,X}$, over the set
of all $g\in \mathcal{G}$ for which there exists $h\in \mathcal H$ satisfying $h\circ g = \Id_d$.

While our main results apply to more general IPMs,
we primarily focus on the $\wass_1$ distance. This choice of $\sf d$ is particularly relevant due to its interpretation as an optimal transport distance. It has found successful applications in various fields, such as computer vision, economics and biology \citep{Ollivier_Pajot_Villani_2014,PeyreCuturi}.

\begin{proposition}\label{thm:lower-bound-hard}
    Let $\bX_1, \ldots, \bX_n \in [0,1]^D$.
    For any $g:[0,1]^d \to [0,1]^D$ having an $L_{\mathcal H}$-Lipschitz-continuous
    left inverse (that is there exists $h:\mathbb R^D \to \mathbb R^d$
    such that $h \circ g = \Id_d$), it holds that
    \begin{align}\tag{LB-hard}
        \wass_1(g \sharp\,\mathcal U_d ,Q) \geq \frac{1}{2L_{\mathcal H}(1 + (2{\sf V}_d n)^{1/d})},
        \label{eq:lb1}
     \end{align}
     where $Q$ is any probability satisfying $Q(\{\bX_1,\ldots,\bX_n\})=1$ and\footnote{${\sf V}_d$ is
     the volume of the unit ball in $\mathbb R^d$.} ${\sf V}_d = \pi^{d/2}/\Gamma(1+d/2)$.
\end{proposition}

It is well-known \cite{dudley_1969} that the $\wass_1$
distance between the empirical and the true distribution is
$O(n^{-1/d})$. It has also been recently proved that the
rate $n^{-1/d}$ is optimal \citep{tang2022minimax}, in the sense that it cannot be improved by any other learned distribution. Thus, if a
generator $\tilde g_n$ is rate-optimal, the triangle
inequality yields $\wass_1(\tilde g_n\sharp \mathcal U_d, P_{n,X}) \leqslant \wass_1(\tilde g_n\sharp \mathcal U_d, P^*) + \wass_1(P^*, P_{n,X})=O(n^{-1/d})$. Thus, if a learned
generator is rate-optimal, its maximal distance from the empirical distribution $P_{n,X}$ of the training set is
of order $n^{-1/d}$. This means, in view of our result,
that if a rate-optimal learned generator has a smooth
left inverse, it lays at a maximal distance from $P_{n,X}$.

To the best of our knowledge, \Cref{thm:lower-bound-hard} provides the
first mathematical quantification of the diversity of examples generated
by the generator. It not only demonstrates that the generator deviates
maximally from the empirical distribution but also shows that it
maintains the same minimum distance from any distribution concentrated
on $n$ points.

\subsection{The Case of Soft Constraint: $\lambda\in(0,\infty)$}

We now turn our attention to measuring the dissimilarity between the empirical distribution and
the generator obtained by \ref{LIPERMe} when $\lambda <
+\infty$. We refer to this scenario as the case of a soft
constraint on left-invertibility. Recall that the
introduction of a penalized version of
the optimization problem, as opposed to the hard
constraint, aims to enhance computational tractability and
facilitate implementation using available tools.
On the downside, the lower bound on the
distance from the empirical distribution, as
presented in the following theorem, is slightly weaker
than that for the constrained generator.

\begin{theorem}\label{thm:lower-bound-pen}
    Let $\bX_1,\dots, \bX_n \in [0,1]^D$ and $\mathcal{H} \subseteq \textup{Lip}_{L_{\mathcal H}}([0,1]^D\to [0,1]^d)$.
    For any $\lambda>0$,
    \begin{align}
        \wass_1(\widehat{g}_{n, \eps}\sharp\,\mathcal{U}_d&, P_{n,X}) \ge \frac{1}{2L_{\mathcal H}(1+(2{\sf V}_dn)^{1/d})}\  - \\
        &\frac1{L_{\mathcal H}\lambda^{1/q}}  \Big( \inf_{g \in \mathcal{G}_{\mathcal H}} \wass_1(g\sharp\,\mathcal{U}_d, P_{n,X}) + \eps \Big)^{1/q},
    \end{align}
    where  $\mathcal{G}_{\mathcal H} =
    \{g\in\mathcal G:\textup{pen}(g) = 0\}$.
\end{theorem}
\begin{corollary}
    If $\lambda $ is chosen so that the inequality
    \begin{align}\label{eq:cor1}
        \lambda \geq 8^q (1+ (2{\sf V}_dn)^{q/d})
        \inf_{g \in \mathcal{G}_{\mathcal H}}\mathbb E\big[
        \wass_1(g\sharp\, \mathcal{U}_d, P_{n,X})\big],
    \end{align}
    holds true, then
    \begin{align}
        \mathbb E\big[\wass_1(\widehat{g}_{n}\sharp\,\mathcal{U}_d, P_{n,X})
        \big]\ge \frac{1}{4L_{\mathcal H}(1+(2{\sf V}_dn)^{1/d})}.
    \end{align}
\end{corollary}

The corollary tells us that if the penalty $\lambda$ is not too
small, any generator satisfying (\ref{LIPERMe}) has a
deviation of the order $n^{-1/d}$ from the empirical distribution.
To gain some understanding of how restrictive this constraint
on $\lambda$ is, let us note that if there is a $L^*$-Lipschitz-continuous $g^*\in \mathcal{G}_{\mathcal H}$
such that $\wass_1(P^*,g^* \sharp\,\mathcal U_d)\leqslant
\sigma^*$, one can check that\footnote{See \Cref{AppB:4}}
\begin{align}
    \inf_{g\in \mathcal{G}_{\mathcal H}}\mathbb E\big[
\wass_1(g\sharp\, \mathcal{U}_d, P_{n,X})\big]
&\leq  \frac{cL^*\sqrt{d}}{n^{1/d}} + \sigma^*
\label{eq:6}
\end{align}
where $c$ is a universal constant.

\begin{remark}
In view of (\ref{eq:cor1}) and (\ref{eq:6}),
when $\sigma^* = 0$ and $q=2$, choosing the penalty
parameter $\lambda$ larger than $C_d n^{1/d}$---
for a constant $C_d$ that depends only on the
dimension of the latent space---
is enough to guarantee that the generator
\ref{LIPERMe} will significantly deviate from the
empirical distribution.
\end{remark}

Before closing this section, let us note that the
inspection of the proof shows that the claim of the
last theorem holds true if we replace $\wass_1$ by
any other distance dominating $\wass_1$. In particular,
the claim is true for $\wass_2$ and for $\sf d_{\mathcal{
F}}$ with any $\mathcal F$ containing all the
1-Lipschitz functions.

\begin{figure*}[htbp]
        \includegraphics[width= \textwidth]{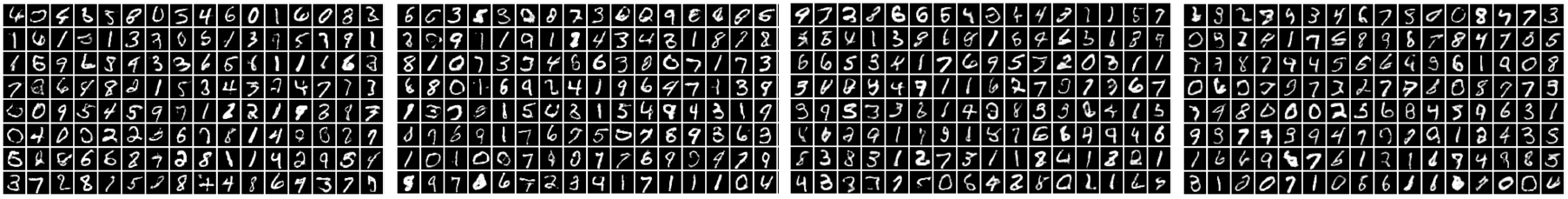}
        \vspace*{-20pt}
        \caption[Handwritten digits generated for different values of $\lambda$.]
        {\small Handwritten digits generated by LIPERM,
        from left to right: $\lambda = 0,1,4,8$. See Fig.~\ref{fig:mnist2}
        for higher-resolution images.}
        \label{fig:mnist0}
\end{figure*}

\begin{figure}[htbp]
        \includegraphics[width=0.49\linewidth]{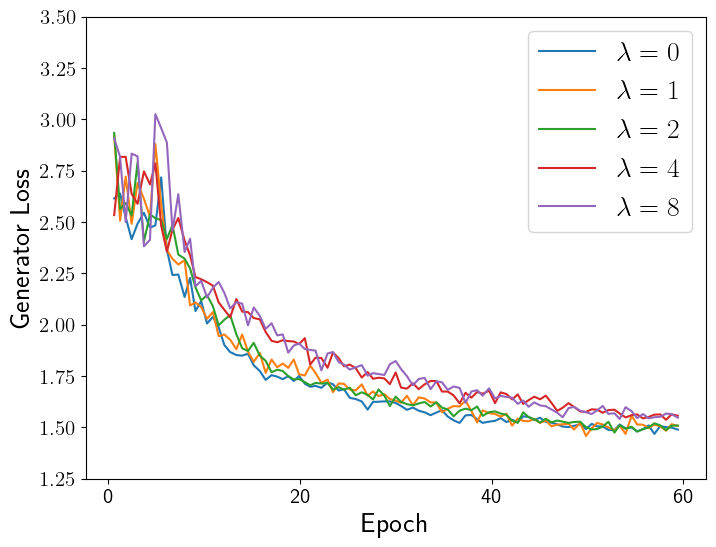}
        \includegraphics[width=0.49\linewidth]{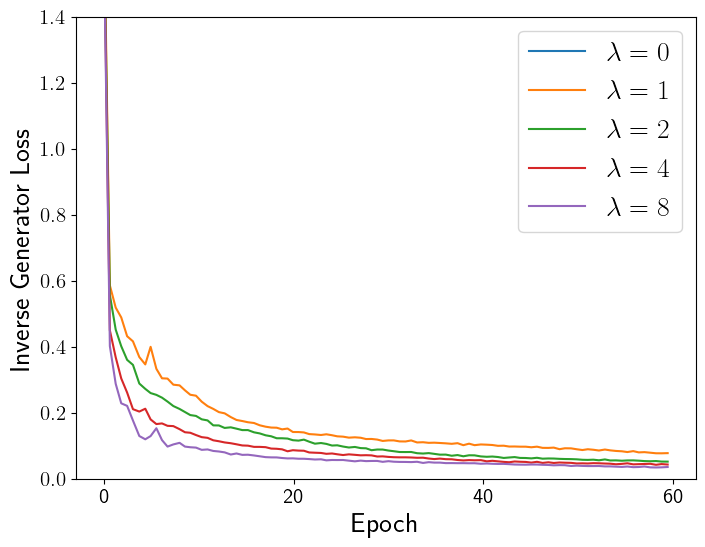}
    \vspace*{-10pt}
    \caption[Handwritten digits generated for different values of $\lambda$.]
    {\small LIPERM on MNIST data. The behavior of the
    generator loss and of the
    left-inverse penalty across the iterations.}
    \label{fig:mnist1}
\end{figure}

\section{Precision of
Left-Inverse-Penalized ERM}\label{sec:3}

In this section, we assess the precision of a generator satisfying (\ref{LIPERMe}) for ${\sf d} = {\sf d}_{\mathcal{F}}$, the integral probability metric (IPM) based on a set of test functions $\mathcal{F}$.
We introduce two parameters $\sigma^*$ and $L^*$, that
quantify the ``manifold assumption''. More precisely, we
 say that $P^*$ satisfies assumption \textsf{A}$(L^*,\sigma^*)$ if
\begin{align}
\label{AssA}\tag*{{\textsf{A}\text{$(L^*,\sigma^*)$}}}
\begin{matrix}
    \exists\,g^*\in \text{Lip}_{L^*}\big([0,1]^d\to[0,1]^D\big)\\
    \text{such that}\quad \wass_1(g^*{\sharp}\,\mathcal U_d,P^*)\leqslant \sigma^*.
\end{matrix}
\end{align}
This assumption accommodates both well-specified and mis-specified scenarios. The former corresponds to \textsf{A}$(L^*,0)$, where $P^*$ is the push-forward of $\mathcal U_d$ by an $L^*$-Lipschitz-continuous function.
The latter corresponds to the case \textsf{A}$(L^*,\sigma^*)$ with $\sigma^*>0$, the parameter $\sigma^*$ indicating the degree of the departure
from a well-specified setting.

\begin{theorem}\label{thm:upper}
    Assume that $\mathcal{F}\subseteq \textup{Lip}_1(\mathbb R^D\to\mathbb R)$  and the dimension of the latent space satisfies $d > 2$. If the observations $\mathbf{X}_1, \ldots, \mathbf{X}_n$ are i.i.d., drawn from $P^*$ satisfying \AssA{L^*}{\sigma^*} for some $L^*, \sigma^*$, then \textup{\ref{LIPERMe}} $\widehat{g}_{n, \eps}$ satisfies
    \begin{align}
        \mathbb E[{\sf d}_ {\mathcal F}(\hat g_{n, \eps}
        \sharp\,\mathcal U_d,P^*)] \leqslant &\inf_{g
        \in \mathcal{G}}  \big\{ {\sf d}_{\mathcal{F}}(g\sharp\,
        \mathcal{U}_d, P^*) + \lambda\, \textup{pen}_{\mathcal H}
        (g) \big\}\\
        &+ 4\sigma^* + \eps + \frac{cL^*\sqrt{d}}{n^{1/d}},
        \label{eq:upper}\tag{UB}
    \end{align}
    where $c>0$ is a universal constant.
\end{theorem}

\begin{remark}
If we assume that the oracle $g^*$ appearing in
\AssA{L^*}{\sigma^*} has a left-inverse that is $L_{\mathcal{H}}$-Lipschitz-continuous, the upper
bound provided by \Cref{thm:upper} becomes $5\sigma^* + {cL^*\sqrt{d}\,}{n^{-1/d}}$. As proven in \citep{schreuder_brunel_dalalyan_2021,tang2022minimax}, this upper bound is minimax-rate-optimal. Notably, the ambient dimension does not appear in the upper bound.
\end{remark}

\begin{remark}
    The assumption $d>2$ is not crucial for deriving an  upper bound similar to (\ref{eq:upper}). In case of $d\in\{1,2\}$, slight modifications occur in the last term: for $d=1$, the denominator becomes $n^{1/2}$, and for $d=2$, an additional $\log n$ factor appears in the numerator. These adjustments are derived by combining our proof, detailed in \Cref{appA:1}, with the corresponding approximation bounds for the uniform distribution in $\wass_1$ distance, addressing the cases of $d\in\{1,2\}$.
\end{remark}

\begin{remark}
    When the true distribution $P^*$ exhibits low sample diversity, \ref{LIPERMe} struggles to provide a precise approximation. In this scenario, the upper bound in \Cref{thm:upper} tends to be large, as indicated by the constant $\sigma^*$.
\end{remark}

\section{Handling Functional Approximations}

The generator (\ref{LIPERM}) and its approximate version (\ref{LIPERMe}) rely on the functional classes $\mathcal G$ and $\mathcal H$. Opting for smaller, parametric classes enhances computational efficiency. Yet, for minimizing the bias term in (\ref{eq:upper}), it is essential to choose $\mathcal G$ and $\mathcal H$ as large as possible. This prompts the question: what if, during training, we substitute $\mathcal G$ and $\mathcal H$ with smaller sets $\mathcal G_0$ and $\mathcal H_0$ possessing good approximation properties? Neural networks, acknowledged as universal approximators for smooth functions \citep{yarotsky2017error,petersen2018optimal, nakada2020adaptive}, are compelling candidates for $\mathcal G_0$ and $\mathcal H_0$.

This consideration extends to the functional class $\mathcal F$. While we aim to gauge the precision of a generator using an IPM with a broad $\mathcal F$, replacing $\mathcal F$ with a smaller set during the training process decreases computational complexity. The next result shows
the impact of replacing $\mathcal F,\mathcal G$ and $\mathcal H$ in (\ref{LIPERMe}) by smaller approximation classes.

\begin{proposition}\label{prop:3}
    Let $\mathcal{F}\subseteq \textup{Lip}_1(\mathbb R^D\to\mathbb R)$  and  $d > 2$. Let $\textup{pen}$ be defined by (\ref{LIP}) with
    $q=2$. Assume that observations $\mathbf{X}_1, \ldots, \mathbf{X}_n$ are i.i.d.\ and drawn from $P^*$, satisfying \AssA{L^*}{\sigma^*} for some $L^*, \sigma^* > 0$. Let $\mathcal F_0\subseteq\mathcal F$, $\mathcal G_0\subseteq \mathcal G$
    and $\mathcal H_0\subseteq \mathcal H$ be functional classes satisfying
    \begin{align}
        &\sup_{f\in\mathcal F}\inf_{f_0\in \mathcal F_0}
        \|f_0 - f\|_\infty\leqslant \delta_{\mathcal F},
        \label{deltaF}\\
        &\sup_{g\in\mathcal G}\inf_{g_0\in \mathcal G_0}
        \|g_0 - g\|_\infty\leqslant \delta_{\mathcal G},
        \label{deltaG}\\
        &\sup_{h\in\mathcal H}\inf_{h_0\in \mathcal H_0}
        \|h_0 - h\|_\infty\leqslant \delta_{\mathcal H}.
        \label{deltaH}
    \end{align}
    Then, any learned generator $\widehat{g}^0 =\widehat{g}^0_{n, \eps}$ satisfying
    \begin{align}
        \hat L_n^{{\sf d}_{\mathcal F_0},\mathcal H_0}
        ({\hat g}^0_{n,\eps}) \leqslant \min_{g\in
        \mathcal G_0}
        \hat L_n^{{\sf d}_{\mathcal F_0},\mathcal H_0}
        (g) + \eps,
    \end{align}
    also satisfies
    $\displaystyle \hat L_n^{{\sf d}_{\mathcal F},\mathcal H}
        ({\hat g}^0_{n,\eps}) \leqslant \min_{g\in
        \mathcal G}\nolimits
        \hat L_n^{{\sf d}_{\mathcal F},\mathcal H}
        (g) + \eps + \delta$,
    where $\delta =
    2\delta_{\mathcal F} + 2\sqrt{d}\,\delta_{\mathcal H}
    +\big(1+ 2\lambda  \sqrt{d}\, L_{\mathcal H }\big)
    \delta_{\mathcal G}$.
\end{proposition}
Combining the results of this proposition with \Cref{thm:upper}, we obtain the following property: if
$\mathcal F_0$, $\mathcal G_0$ and $\mathcal H_0$ are
chosen so that $\delta_{\mathcal F} \leqslant \eps/6$,
$\delta_{\mathcal G} \leqslant \eps/3(1+2\lambda \sqrt{d}\,
L_{\mathcal H})^{-1}$ and $\delta_{\mathcal H}\leqslant
d^{-1/2}\eps/6$, see (\ref{deltaF}-\ref{deltaH}), then
any $\eps$-minimizer ${\hat g}^0$ of $\hat L_n^{\sf d_{\mathcal F_0},\mathcal H_0}$ over $\mathcal G_0$ satisfies
\begin{align}
    \mathbb E[{\sf d}_ {\mathcal F}({\hat g}^0
    \sharp\,\mathcal U_d,P^*)] \leqslant &\inf_{g
    \in \mathcal{G}}  \big\{ {\sf d}_{\mathcal{F}}(g\sharp\,
    \mathcal{U}_d, P^*) + \lambda\, \textup{pen}_{\mathcal H}
    (g) \big\}\\
    &+ 4\sigma^* + \frac{cL^*\sqrt{d}}{n^{1/d}}
    + 2\eps.
\end{align}
This suggests that with appropriately chosen approximation classes, the learned generator can approach the performance of the best generator from the class of all Lipschitz-continuous functions \(g\) with a Lipschitz-continuous left-inverse \(h\). Notably, utilizing a set of ``critics'' \(\mathcal F\) consisting of neural networks that are \(\delta\)-approximations of 1-Lipschitz functions results in an error within \(2\delta\) of the error obtained when trained with \(\mathcal F = \textup{Lip}_1\). Note that according to \citep[Theorem~1]{yarotsky2017error}, one can achieve an
error $\delta_{\mathcal F}$ for $\mathcal F$ the functions
with derivatives bounded by $1$ using ReLU activated
neural networks of depth $O(\log(1/\delta_{\mathcal F}))$
and of the number of weights and units $O(
\delta^{-D}_{\mathcal F}\log(1/\delta_{\mathcal F}))$.
Furthermore, in view of \Cref{thm:lower-bound-pen},
\begin{align}
        \wass_1(\widehat{g}^0
        \sharp\,\mathcal{U}_d&, P_{n,X}) \ge \frac{1}{2L_{\mathcal H}(1+(2{\sf V}_dn)^{1/d})}\  - \\
        &\frac1{L_{\mathcal H}\lambda^{1/2}}  \Big( \inf_{g \in \mathcal{G}_{\mathcal H}} \wass_1(g\sharp\,\mathcal{U}_d, P_{n,X}) + 2\eps \Big)^{1/2},
    \end{align}
where $c>0$ is a universal constant.

Finally, note that the last proposition can be used
with $\mathcal G = \mathcal G_0$, a parametric set defined
by neural networks with a given architecture. Then
$\delta_{\mathcal G} =0$, which conveniently simplifies
the expression of $\delta$. In addition, the approximation
error may be directly bounded using results from
\citep{yang2022capacity,lu2020universal}.

\section{Numerical Experiments}
\label{sec:6new}

In this section, we aim to evaluate the performance
of the left-inverse-penalized WGANs. Our implementation\footnote{Our
code uses the framework of \cite{labml} and is available
\href{https://github.com/TigranGalstyan/LIPERM_annotated_deep_learning_paper_implementations/tree/liperm/labml_nn/gan/wasserstein/gradient_penalty/liperm}{here}.
}
follows the pseudo-code presented in Algorithm \ref{alg:cap},
and is inspired by the code accompanying \citep{GulrajaniAADC17},
where WGANs with gradient penalty on the discriminator/critic
network is discussed. We add the LIPERM penalization to the
objective function of the WGANs. All the functional classes
$\mathcal F_0$, $\mathcal G_0$, and $\mathcal H_0$ are represented
by neural networks, the architectures of which are presented in
the supplementary material. In all our experiments, we chose
$n_{critic} = 5$ and $\gamma= 1$
.
\begin{algorithm}[h]
    \caption{WGAN-LIPERM. We take $\lambda \in \{0, 1, 4, 8\}$,
    }\label{alg:cap}
    \begin{algorithmic}[1]
        \REQUIRE LIP coefficient $\lambda$, gradient penalty
        coefficient $\gamma$, number of iterations $N_{\rm iter}$,
        number of critic iterations per generator iteration
        $n_{\text{critic}}$, batch size $m$
        \REQUIRE initial critic and generator parameters
        $(\bm w_0,\bm \theta_0)$, initial left inverse
        network parameters $\bm\phi_0$, $k=0$.
        \REPEAT
        \STATE{$k\gets k+1$}
        \FOR{$t = 1, ..., n_{\text{critic}}$}
        \FOR{$i = 1,...,m$}
            \STATE Draw $\bm x\sim P_{n,X}$\quad {\color{olive}
            (true examples)}
            \STATE Draw $\bm u \sim \mathcal{U}_d$\quad\quad {\color{olive}
            (latent variables) }
            \STATE Draw $\epsilon \sim \mathcal U [0, 1]$
            \STATE $\tilde{\bm x} \gets G_{\bm\theta} (\bm u)$
            \quad\quad {\color{olive} (generated examples)}
            \STATE $\hat{\bm x} \gets \epsilon \bm x +
            (1 - \epsilon) \tilde{\bm x}$
            \STATE $L_d \gets F_{\bm w}(\tilde{\bm x}) -
            F_{\bm w}(\bm x)$
            \STATE $L_d^{(i)} \gets L_d + \gamma (\|\nabla_{\hat{
            \bm x}} F_{\bm w}(\hat{\bm x})\|^2 - 1)^2$
        \ENDFOR
        \STATE $\bm w \gets \text{Adam}(\nabla_w \frac{1}{m}
        \sum_{i=1}^m L_d^{(i)},\{\bm w\})$
        \ENDFOR
        \FOR{$i = 1,...,m$}
            \STATE Draw $\bm u \sim \mathcal{U}_d $
            \STATE $L_g^{(i)} \gets - F_{\bm w} (G_{\bm \theta}
            (\bm u)) + \lambda
            \|H_{\bm\phi}(G_{\bm\theta}(\bm u)) - \bm u\|^2$
        \ENDFOR
        \STATE $({\bm \theta}, {\bm\phi}) \gets
        \text{Adam}(\nabla_{\bm\theta, \bm\phi} \frac{1}{m}
        \sum_{i=1}^m L_g^{(i)} , \{(\bm\theta, \bm\phi)\})$
        \UNTIL{$k>N_{\rm iter}$}
    \end{algorithmic}
\end{algorithm}



We conducted experiments on three widely used datasets: Swiss Roll,
MNIST and CIFAR 10. The results are briefly summarized in this
section. The main messages of these experiments are that (a) it
is possible to implement the LIPERM algorithm and to get generators
that are nearly invertible, that is they have a small left inverse
penalty, (b) the visual quality of the results does not deteriorate
when the penalty parameter $\lambda$ is increased. As for the non
replication, it seems that with the architectures and optimizers used
in the standard data sets considered in this section, the replication
or the lack of creativity is not an issue. Therefore, we could not
observe an increase in creativity due to introducing the left
inverse penalty.

\textbf{Swiss Roll \citep{swissroll}:}
The Swiss Roll dataset consists  of 2D points arranged in a rolled
structure, corrupted by a 2D Gaussian noise with some standard
deviation $\sigma$. We used three values of $\sigma = 3/2,3/4,3/8$
We used a training set of size 1000, a batch-size of 200 and run
experiments for each value of $\lambda$ from  $\{0,1,4,8\}$.
Fig.~\ref{fig:swiss_roll_gen0}
and Fig.~\ref{fig:swiss_large_sigma} in Appendix show the training set
and the points generated by the learned distribution.  These results
are consistent with the theory: for increasing but not very large values
of $\lambda$ the accuracy of the generator is preserved. Note
that it is known \citep{VEEGAN} that training a generator for this
data is highly unstable. Plots in Fig.~\ref{fig:swiss_roll_gen2}
confirm this instability and show that, unfortunately, the left-inverse
penalty does not alleviate it. We also conducted additional unreported
experiments, training for over 4,000 epochs (up to 20,000 epochs),
but observed no improvement in the results.
\begin{figure*}[htbp]
        \centering
            \begin{subfigure}[b]{0.24\textwidth}
            \centering
            \includegraphics[width=0.98\textwidth]{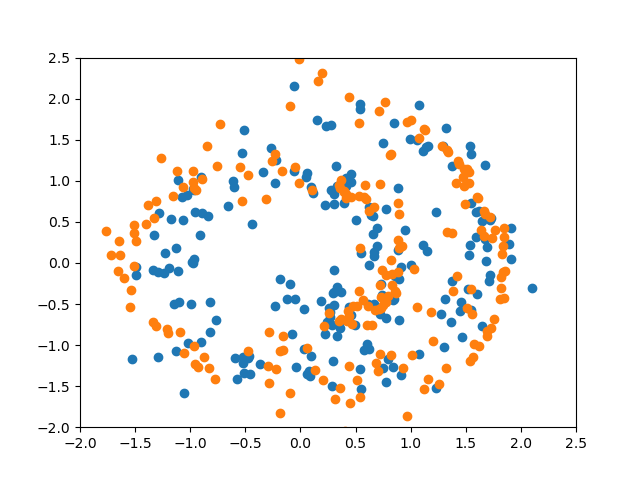}
        \end{subfigure}
        \hfill
        \begin{subfigure}[b]{0.24\textwidth}
            \centering
            \includegraphics[width=0.98\textwidth]{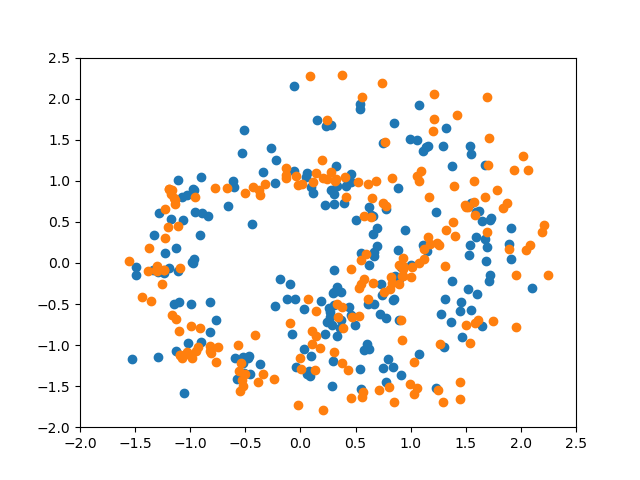}
        \end{subfigure}
        \hfill
        \begin{subfigure}[b]{0.24\textwidth}
            \centering
            \includegraphics[width=0.98\textwidth]{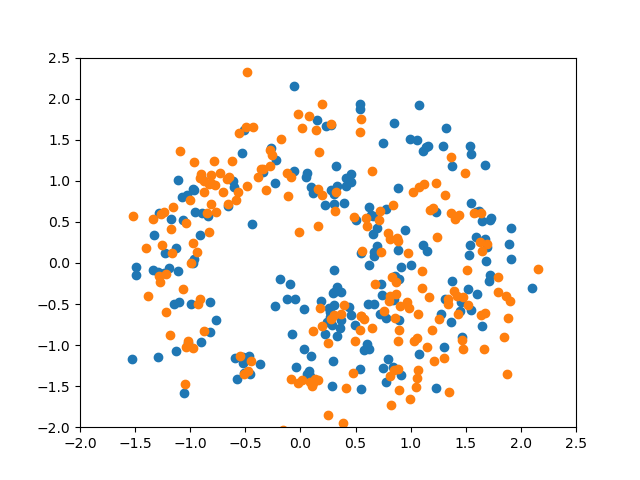}
        \end{subfigure}
        \hfill
        \begin{subfigure}[b]{0.24\textwidth}
            \centering 
            \includegraphics[width=0.98\textwidth]{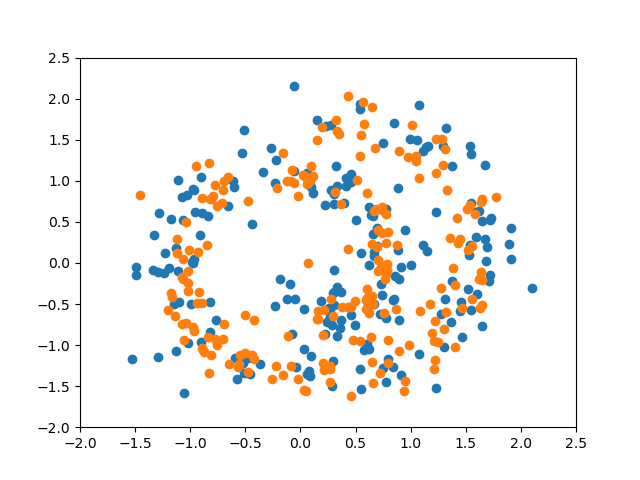}
        \end{subfigure}
        \vspace{-10pt}
        \caption[ $\lambda$.]
        {\small Swiss Roll: samples generated by LIPERM WGAN
        (blue) and original data (orange) for
        $\lambda = 0,1,4,8$ (left to right).}
        \label{fig:swiss_roll_gen0}
\end{figure*}

\textbf{MNIST \citep{lecun1998mnist}:}
MNIST dataset  is commonly used for handwritten digit recognition.
In this paper, we train WGAN plus LIPERM penalty on MNIST and
generate images for different values of LIPERM coefficient $\lambda
\in \{0, 1, 4, 8\}$. As shown in Fig.~\ref{fig:mnist1}, the
training process seems to converge after nearly 60 epochs, when
the batch size is 512. The architectures of the networks used
in this experiment are presented in Table \ref{table:MNIST_architecture}
in Appendix.

The results depicted in Fig.~\ref{fig:mnist0} suggest
that the generated images when $\lambda = 1,4,8$ are as good as those for the vanilla
WGAN ($\lambda = 0$). Furthermore, the right plot of Fig.~\ref{fig:mnist1} shows
that the implementation we used effectively minimizes the LIP and that the final
result is almost left invertible.

\begin{figure*}
        \begin{subfigure}[b]{0.32\textwidth}
            \centering
            \includegraphics[width=0.95\textwidth]{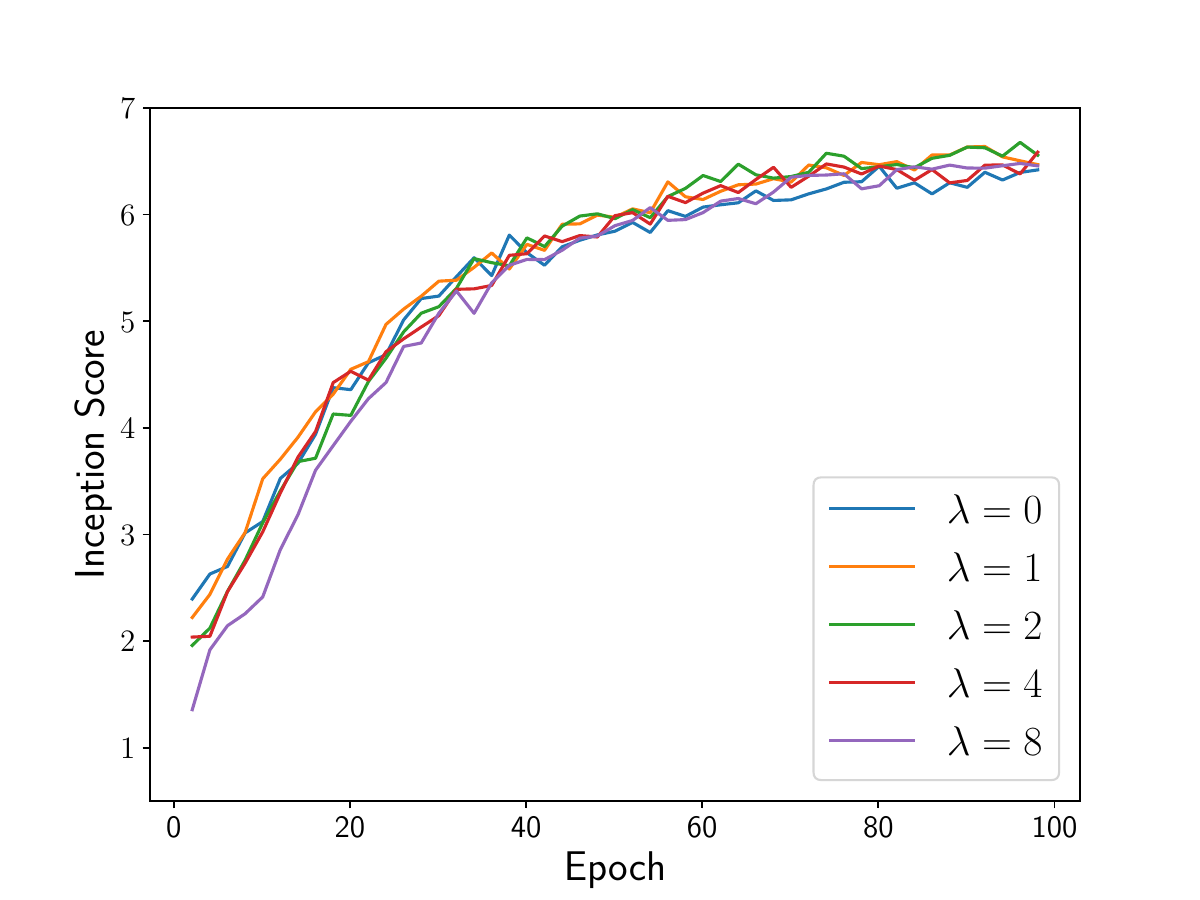}
        \end{subfigure}
        \hfill
        \begin{subfigure}[b]{0.32\textwidth}
            \centering
            \includegraphics[width=0.95\textwidth]{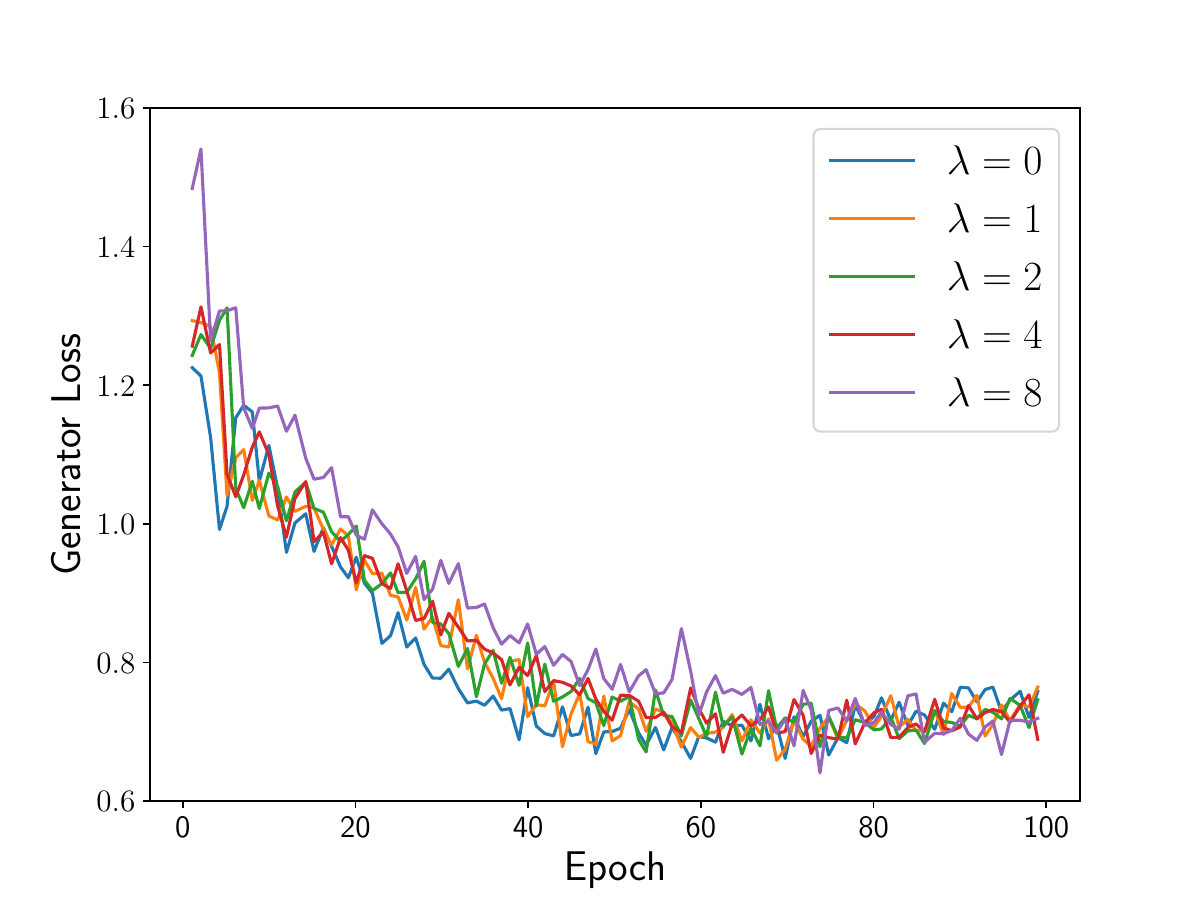}
            \label{fig:original}
        \end{subfigure}
        \hfil
        \begin{subfigure}[b]{0.32\textwidth}
            \centering
            \includegraphics[width=0.95\textwidth]{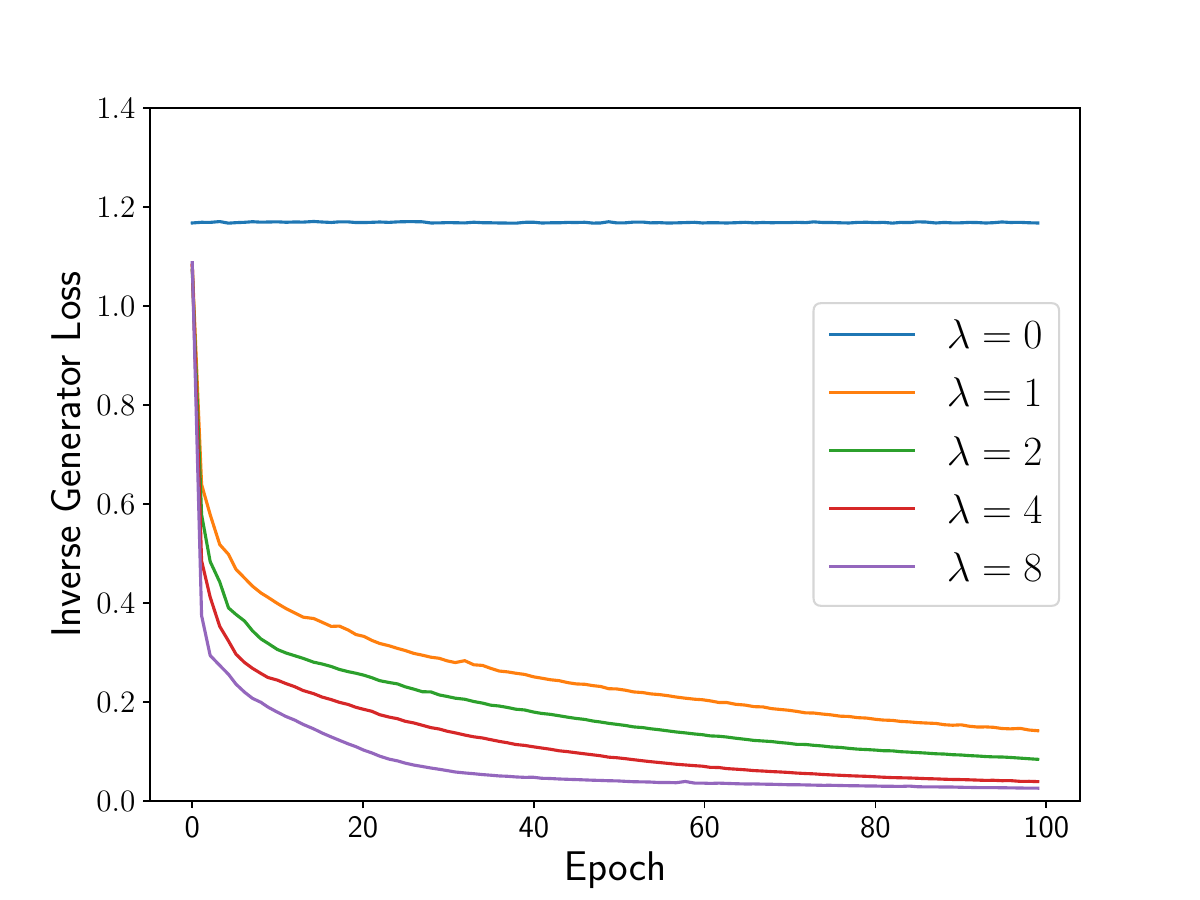}
        \end{subfigure}
        \caption{Experimental results on CIFAR-10 data set. Left: the
        evolution of the Inception Score across the iterations. Middle:
        the evolution of the generator loss across the iterations, for
        various values of $\lambda$. Right: the evolution of the
        left-inverse penalty across the iterations, for various values
        of  $\lambda$.}
        \label{fig:cifar_fid_is0}
\end{figure*}

\textbf{CIFAR-10 \citep{krizhevskyCifar}:}
Finally, to demonstrate that our method can also be applied
on real-world images, we perform experiments on CIFAR-10 dataset.
For Generator network we employ ResNet-style architecture and
for Discriminator and Inverse-Generator we use simpler
convolutional networks. For more architectural details,
please refer to our implementation.

Results and generated
samples are presented in \Cref{fig:cifar_fid_is0} and
\Cref{fig:cifar_examples}. One can observe that increasing
$\lambda$ does not decrease the quality of generated samples.
Indeed, the inception scores as well as the generator losses
depicted in \Cref{fig:cifar_fid_is0} seem to show that the
choice of $\lambda$ has almost no impact on the accuracy.
However, according to the right plot of the same figure,
the generator corresponding to $\lambda=8$ has the smallest
penalty and, therefore, is closer to be left invertible.
\begin{figure*}[htbp]
        \centering
        \includegraphics[width = 0.9\textwidth]{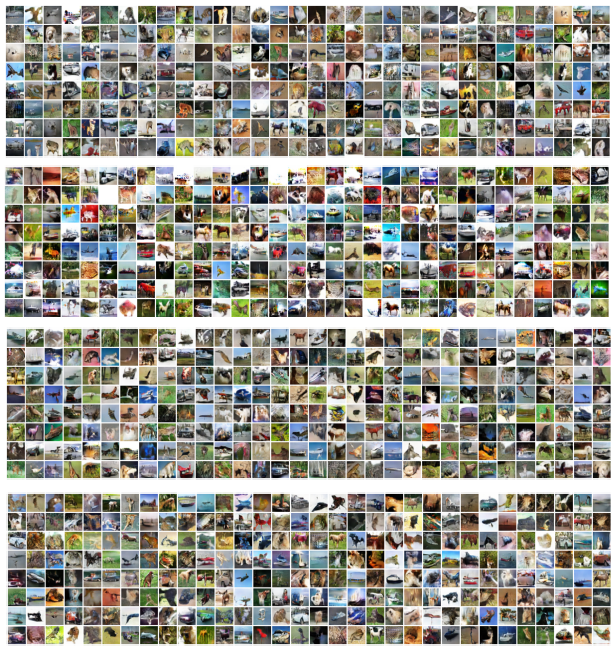}
        \vspace{-10pt}
        \caption{Generated image samples on CIFAR-10 for different values of \ref{LIPERMe} $\lambda =0$; $\lambda =1$; $\lambda=4$, and $\lambda=8$}
        \label{fig:cifar_examples}
\end{figure*}

\section{Summary, Conclusion and Limitations}
\label{sec:5}

In this paper, we have presented a theoretical analysis
of training generative models with two key properties: avoidance of replication of the observed examples and convergence to the true distribution at a minimax optimal rate. Our main contribution is that the existence of a smooth left-inverse implies the first one of these properties. We further introduced the left-inverse-penalized empirical risk minimization \LIPERM1 framework, with a penalty encouraging the generator to possess a smooth inverse. We showed, both theoretically and empirically, that \LIPERM1 and its
approximate version \ref{LIPERMe} enjoy the mentioned desirable properties.

\textbf{Limitations}  The incorporation of left invertibility poses certain computational challenges. While our numerical experiments indicate that these difficulties are not insurmountable, optimization errors will likely be dominant in most applications. Our work does not explicitly handle choosing $d$,
the dimension of the latent space. One approach to address this is considering the Bourgain theorem \citep{Bourgain}, as in \citep{xiao2018bourgan}.
The choice of $d$ may also influence the Lipschitz constant $L^*$, which plays a pivotal role. The interplay between these quantities needs to be better understood;
the methodology developed in \citep{Jordan20,JordanD21,WangM23} might help in this task.
Extending our analysis to functional classes with higher smoothness and diffusion models is non-trivial. On a related note, considering distances between distributions that are not IPMs, such as the Sinkhorn divergence \cite{GenevayPC18,Luise2020}
cannot be done using the methodology of this paper.
These challenges pose interesting questions for future research.

\section*{Impact Statement}
This paper presents work that aims to advance the field of Machine Learning. There are many potential societal consequences of our work, none which we feel must be specifically highlighted here.

\section*{Acknowledgments}
This work was supported by the grant ANR-11-IDEX0003/Labex Ecodec/ANR-11-LABX-0047, and the ADVANCE Research Grant provided by the Foundation for Armenian Science and Technology (FAST) and Yerevan State University (YSU).


\bibliography{ICML2024_refs}
\newpage

\appendix
\onecolumn

\tableofcontents

\addtocontents{toc}{\protect\setcounter{tocdepth}{3}}

\section*{Appendix}\label{app}

The purpose of this appendix is twofold: to present the proofs of all the mathematical claims presented in the main paper and to provide some additional experimental results. We illustrate the effect of the penalization parameter $\lambda$ on the trained generator using the \ref{LIPERMe} framework. The code
that can be used to reproduce all the experiments can be found here \url{https://github.com/TigranGalstyan/LIPERM_annotated_deep_learning_paper_implementations/tree/liperm/labml_nn/gan/wasserstein/gradient_penalty/liperm}.

\section{Proof of the upper bound on the risk}

The proof of \Cref{thm:upper}, provided below, relies on repeated use of the triangle inequality, the near-minimization property of \ref{LIPERMe}, as well as on exploiting Lipschitz continuity assumptions. The final bound is derived by combining these inequalities with the approximation bound
in the Wasserstein-$1$ distance of the uniform distribution
by its empirical counterpart.

This result demonstrates that under mild assumptions, the generator trained by \ref{LIPERMe}
achieves the optimal rate for any choice of $\lambda$. Therefore, there is flexibility in selecting $\lambda$ to enforce dissimilarity with the training examples without compromising accuracy.

\subsection{Proof of \Cref{thm:upper}}\label{appA:1}

Let $g$ be an arbitrary element from $\mathcal{G}$. The proof begins by using the triangle inequality multiple times, resulting in the following sequence of inequalities:
\begin{align}
   {\sf d}_{\mathcal{F}}(\widehat{g}_{n, \epsilon}\sharp\, \mathcal{U}_d, P^*) & \leq  {\sf d}_{\mathcal{F}}(\widehat{g}_{n, \epsilon}\sharp\, \mathcal{U}_d, P_{n,X}) + {\sf d}_{\mathcal{F}}(P_{n,X}, P^*) \nonumber \\
   & = {\sf d}_{\mathcal{F}}(\widehat{g}_{n, \epsilon}\sharp\, \mathcal{U}_d, P_{n,X}) + \lambda\, \textup{pen}_{\mathcal{H}}(\widehat{g}_{n, \epsilon}) - \lambda\, \textup{pen}_{\mathcal{H}}(\widehat{g}_{n, \epsilon}) + {\sf d}_{\mathcal{F}} (P_{n,X}, P^*) \nonumber \\
   & \leq {\sf d}_{\mathcal{F}}(g\sharp\, \mathcal{U}_d, P_{n,X}) + \lambda\, \textup{pen}_{\mathcal H}( g) + \epsilon - \lambda\, \textup{pen}_{\mathcal{H}}(\widehat{g}_{n, \epsilon}) + {\sf d}_{\mathcal{F}} (P_{n,X}, P^*) \nonumber \\
   & \leq {\sf d}_{\mathcal{F}}(g\sharp\, \mathcal{U}_d, P^*) + \lambda\, \textup{pen}_{\mathcal H}( g) + \epsilon - \lambda\, \textup{pen}_{\mathcal{H}}(\widehat{g}_{n, \epsilon}) + 2{\sf d}_{\mathcal{F}} (P_{n,X}, P^*).
\end{align}

In the second inequality, we use the fact that $\widehat{g}_{n,\eps}$ is an $\eps$-minimizer of (\ref{LIPERM}), whereas in the last line, the triangle inequality is applied. As the left-inverse penalty is always nonnegative, the last display implies:
\begin{align}\label{eq:oracle-ineq}
    {\sf d}_{\mathcal{F}}(\widehat{g}_{n, \epsilon}\sharp\, \mathcal{U}_d, P^*) &\le \inf_{g\in \mathcal{G}} \big\{ {\sf d}_{\mathcal{F}}(g\sharp\, \mathcal{U}_d, P^*) +
    \lambda\, \textup{pen}_{\mathcal H}(g) \big\} + 2{\sf d}_{\mathcal{F}} (P_{n,X}, P^*) + \epsilon \\
    &\le \inf_{g \in \mathcal{G}}  \big\{ {\sf d}_{\mathcal{F}}(g\sharp\, \mathcal{U}_d, P^*) + \lambda\, \textup{pen}_{\mathcal H}
    (g) \big\} + 2{\sf d}_{\mathcal{F}} (P_{n,X}, g^*\sharp\mathcal U_d) + 2\sigma^* + \epsilon.\label{eq:8}
\end{align}
The last line follows from \AssA{L^*}{\sigma^*} and the triangle inequality. By applying \AssA{L^*}{\sigma^*} again, we can establish the existence of independent random variables $\bU_i\sim\mathcal{U}_d$ such that $\mathbb{E}[\|\bX_i-g(\bU_i)\|]\leqslant \sigma^*$, for every $i\in[n]$. This implies that
\begin{align}
    {\sf d}_{\mathcal{F}} (P_{n,X}, g^*\sharp\,
    \mathcal U_d) & = \sup_{f\in\mathcal{F}} \Big|\frac1n \sum_{i=1}^n
    f(\bX_i) - \mathbb E[f\circ g^*(\bU)]\Big|\\
    &= \sup_{f\in\mathcal{F}} \bigg\{\Big|
    \frac1n \sum_{i=1}^n\big(
    f(\bX_i) - f\circ g^*(\bU_i)\big) +
    \frac1n \sum_{i=1}^n\big(f\circ g^*(\bU_i) -
    \mathbb E[f\circ g^*(\bU)]\big)\Big|\bigg\}\\
    &\leqslant \frac1n \sum_{i=1}^n\big\| \bX_i -
    g^*(\bU_i)\big\| + \sup_{\psi \in\text{Lip}_{L^*}}
    \Big|
    \frac1n \sum_{i=1}^n\big(
    \psi(\bU_i) - \mathbb E[\psi(\bU)]\Big|.
\end{align}
Here, we used the fact that $\mathcal{F}\subseteq \text{Lip}_1([0,1]^D\to \mathbb{R})$ and that the composition of a 1-Lipschitz-continuous and an $L^*$-Lipschitz-continuous functions is itself $L^*$-Lipschitz continuous. Taking the expectation and employing the dual formulation of the Wasserstein-$1$ distance, we arrive at
\begin{align}
    \mathbb E[{\sf d}_{\mathcal{F}} (P_{n,X},
    g^*\sharp\, \mathcal U_d)]
    &\leqslant \frac1n \sum_{i=1}^n\mathbb E[\big\| \bX_i -
    g^*(\bU_i)\big\|] + L^*\mathbb E[\wass_1( P_{n,U},
    \mathcal U_d)]\\
    &\leqslant \sigma^* + L^*\mathbb E[\wass_1( P_{n,U},
    \mathcal U_d)].\label{eq:9}
\end{align}
Here, $\hat{P}_{n,U}$ is the empirical distribution of the sample $\bU_1,\ldots,\bU_n$. Finally, using the well-known bound on the error of the empirical distribution in the Wasserstein-$1$ distance  (for example, see \citep[Proposition 1]{weed2022estimation}), we obtain:
\begin{align}\label{wass1}
    \mathbb E[\wass_1( P_{n,U},
    \mathcal U_d)] \leqslant \frac{c\sqrt{d}}{n^{1/d}}.
\end{align}
Combining this bound with (\ref{eq:8}) and
(\ref{eq:9}), we get the stated upper bound.

\section{Proofs for the deviation of the generative distribution from the empirical distribution}

\subsection{Lower bounding the distance between the uniform distribution and any discrete distribution}

\def\ba{\boldsymbol{a}}
\def\by{\boldsymbol{y}}

\begin{proposition}\label{prop:lower_bound}
    Assume that $\mathcal{F}$ contains all the $1$-Lipschitz-continuous functions from $\mathbb{R}^d $ to $\mathbb{R}$. For any set  of points $\ba_1,\ldots,\ba_n\in [0,1]^d$
    and any set of weights $w_1,\ldots,w_n\geq 0$ summing to one, we have
    \begin{align}
        {\sf d}_\mathcal{F}\Big(\mathcal U_d, \sum_{i=1}^n w_i\delta_{\ba_i}\Big)
        \geqslant
        \frac{1}{2 + 2(2{\sf V}_d n)^{1/d}}\ ,
    \end{align}
    where $\mathcal U_d$ is the uniform distribution on $[0,1]^d$
     with ${\sf V}_d = \frac{\pi^{{d}/{2}}}{\Gamma(\frac{d}{2} +1)}$ being the volume of the unit ball\footnote{$\Gamma$ is Euler's gamma function.} in $\mathbb R^d$.
\end{proposition}

\begin{proof}
     Let $A = \{\ba_1,\ldots, \ba_n\}$ and  define the function
        \begin{align}
            g_A:[0,1]^d\to\mathbb [0,1],\qquad g_A(\bx) =
            \min\big(1,\min_{\ba\in A}\|\bx-\ba\|\big), \label{eq:condgA}
        \end{align}
    We start the proof by noticing that for any $A \subseteq \mathbb{R}^d$ the function $g_A$ is $1$-Lipschitz. Therefore, we have $\{g_A: A \subseteq \mathbb{R}^d\} \subseteq \text{Lip}_1 \subseteq \mathcal{F}$. Therefore,
    \begin{align}
        {\sf d}_\mathcal{F}\Big(\mathcal U_d, \sum_{i=1}^n w_i\delta_{\ba_i}\Big) &\geq \sup_A \bigg(\int_{[0,1]^d} g_A(\bx)\,d\bx - \sum_{i=1}^n w_i g_A(\ba_i) \bigg)= \sup_A \int_{[0,1]^d} g_A(\bx)\,d\bx.
    \end{align}
    We denote by $B(\bx_0, r)$ the ball in $\mathbb{R}^d$ with center $\bx_0$ and radius $r$, $B(\bx_0, r) = \{ \bx \in \mathbb{R}^d : \| \bx - \bx_0 \| \le r\}$. Let us define $A_i = B(\ba_i,r)$, for $i=1,\ldots,n$.
    Using this notation we arrive at
    \begin{align}
        \int_{[0,1]^d} g_A(\bx)\,d\bx
        &\geq
        \int_{(\cup_{i=1}^n A_i)^c\cap [0,1]^d} \min\big(1,\,\min_{\by\in A}\|\bx-\by\|\big)\,d\bx\\
        &\ge
        \int_{(\cup_{i=1}^n A_i)^c\cap [0,1]^d} \min\big(1,r \big)\,d\bx.
        \label{eq:minepsn}
    \end{align}

We now state an auxiliary lemma, the proof of which is deferred to the end of the section.
\begin{lemma}\label{lem:2}
Let $S = [0,1]^d$ with $d \in \mathbb{N}$ and let $\mu$ be a probability
measure on $[0,1]^d$ admitting a density (with respect to the Lebesgue measure) bounded by some constant $b<\infty$. For any $r>0$ and $k\in \mathbb N$, if $B_1,\ldots,B_k \subseteq \mathbb R^d$ are balls of radius $r$ such that
\begin{gather}
    \mu\big(B_1 \cup \ldots \cup B_k\big) > {1}/{2}
\end{gather}
then,
\begin{align}
    r^d > \frac{1}{2bk \,{\sf V}_d} \cdot r^{-d},\quad \text{ where }\quad {\sf V}_d = {\rm Vol}(B(\mathbf 0,1)).\label{ineq:L2:1}
\end{align}
\end{lemma}

Let us choose $r = (2n{\sf V_d})^{-1/d}$. This value of $r$
does not satisfy (\ref{ineq:L2:1}). In view of \Cref{lem:2},
this implies that
\begin{gather}
    \mu\big(A_1 \cup \ldots \cup A_n\big)
    \leqslant {1}/{2}.
\end{gather}
Since $\mu = \mathcal U_d$ is a probability measure on $[0,1]^d$,
we get
\begin{gather}
    \mu\big((A_1 \cup \ldots \cup A_n)^c\cap [0,1]^d\big)
    > {1}/{2}.
\end{gather}
Combining this with (\ref{eq:minepsn}), we arrive at
\begin{align}
        {\sf d}_\mathcal{F}\Big(\mathcal U_d, \sum_{i=1}^n w_i\delta_{a_i}\Big)
        \geq (1/2)\min(1,(2{\sf V}_dn)^{-1/d} ).
\end{align}
In other terms, if $n\geq (2{\sf V}_d)^{-1}$, then
\begin{align}
        {\sf d}_\mathcal{F}\Big(\mathcal U_d, \sum_{i=1}^n w_i\delta_{\ba_i}\Big)
        \geq \frac12 (2{\sf V}_d n)^{-1/d} = \frac{1}{2^{(d+1)/d}} ({\sf V}_d n)^{-1/d}  
\end{align}
otherwise
\begin{align}
        {\sf d}_\mathcal{F}\Big(\mathcal U_d, \sum_{i=1}^n w_i\delta_{\ba_i}\Big)
        \geq \frac12.
\end{align}
Putting together the inequalities from the last two displays concludes the proof of the desired lower bound.
\end{proof}

\begin{proof}[Proof of \Cref{lem:2}]
Let $\nu$ be the Lebesgue measure on $[0,1]^d$.
We know that
    $\mu(A) = \int_A \varphi(x)\, \nu(dx)$ with a probability density function $\varphi$ satisfying $0\leq \varphi(x)\leq b$ for all $x\in [0,1]^d$. Therefore,
\begin{align}
    \frac{1}{2} < \mu(B_1 \cup \ldots \cup B_k) & \leq \sum\limits_{j=1}^k \mu(B_j)
     = \sum\limits_{j=1}^k \int_{B_j} \varphi(x)\,\nu(dx)
     \leq \sum\limits_{j=1}^k b \cdot \nu(B_j) \label{eq:2.1}
\end{align}
Moreover, we know that $\nu(B_j) = {\sf V}_d r^d$ for all $j = 1, \dots k$. Combining this inequality with (\ref{eq:2.1}), we get
\begin{align}
    \frac{1}{2} <  k b \, {\sf V}_d \, r^d.
\end{align}
This yields $r^d > \frac{1}{2 bk \, {\sf V}_d}$.
\end{proof}

\subsection{Proof of \Cref{thm:lower-bound-hard}}

    Recall that
    \begin{align}
        {\sf d}_{\mathcal{F}}({g} \sharp\, \mathcal{U}_d, P_{n,X}) = \sup_{f \in \mathcal{F}} \bigg| \int_{[0,1]^d} f({g}(\bu))\,\rmd\bu - \frac{1}{n} \sum\limits_{i=1}^n f(\bX_i) \bigg|.
    \end{align}
    In addition, there exists an $L_{\mathcal H}$-Lipschitz-continuous function $h$ such that $ h\circ g = \Id_d$ and
    \begin{align}
        \| h(\bx)- h(\bx')\|\leqslant L_{\mathcal H}\|\bx- \bx'\|
        \qquad \forall \bx,\bx'\in\mathbb R^D.
    \end{align}
    In order to establish the lower bound on ${\sf d}_{\mathcal{F}}({g} \sharp
    \mathcal{U}_d, P_{n,X})$, we use the fact that $\mathcal F$
    contains all the functions of the form $\psi \circ {h}$, where
    $\psi:[0,1]^d \to \mathbb{R}$ is any $(1/L_{\mathcal H})$-Lipschitz-continuous function.
    Indeed, since $h$ is $L_{\mathcal H}$-Lipschitz, the function $\psi\circ h$
    belongs to $\text{Lip}_1\subseteq \mathcal F$. Therefore,
    \begin{align}
        {\sf d}_{\mathcal{F}}({g} \sharp\, \mathcal{U}_d, P_{n,X})
        &\geqslant \sup_{\psi \in \text{Lip}_{L_{\mathcal H}^{-1}}} \bigg\{ \int_{[0,1]^d} (\psi\circ \underbrace{ h\circ {g}}_{=\Id_d})(\bu)\,\rmd\bu - \frac{1}{n} \sum\limits_{i=1}^n (\psi\circ h)(\bX_i) \bigg\}\\
        & = \sup_{\psi \in \text{Lip}_{L_{\mathcal H}^{-1}}} \bigg\{\int_{[0,1]^d} \psi (\bx)\,\rmd\bx -
        \frac{1}{n} \sum\limits_{i=1}^n \psi(\bZ_i) \bigg\},
    \end{align}
    where we have used the notation $\bZ_i = h(\bX_i)$, for $i=1,\ldots,n$.
    Clearly, $\psi \in \text{Lip}_{L_{\mathcal H}^{-1}}$ is equivalent to
    $L_{\mathcal H}\psi \in \text{Lip}_1$. This implies that
    \begin{align}
        {\sf d}_{\mathcal{F}}({g} \sharp\, \mathcal{U}_d, P_{n,X})
        &\geqslant \frac1{L_{\mathcal H}}\sup_{\psi \in \text{Lip}_1} \bigg\{\int_{[0,1]^d} \psi (\bu)\,\rmd\bu - \frac{1}{n} \sum\limits_{i=1}^n \psi(\bZ_i) \bigg\}.
    \end{align}
    The right-hand side of the inequality above is precisely the
    $\wass_1$ distance between the uniform measure on $[0,1]^d$ and
    the empirical distribution of $\bZ_1, \dots, \bZ_n$. Therefore, we arrive at
    \begin{align}
    {\sf d}_{\mathcal{F}}({g} \sharp\, \mathcal{U}_d, P_{n,X}) \ge \frac{1}{L_{\mathcal H}} \wass_1\Big(\mathcal{U}_d, \frac{1}{n}\sum_{i=1}^n \delta_{\bZ_i}\Big) \ge \frac{1}{2L_{\mathcal H}(1 + (2{\sf V}_d n)^{1/d})},
\end{align}
where the last inequality follows from \Cref{prop:lower_bound}.

\subsection{Proof of \Cref{thm:lower-bound-pen}}

    Let $\hat h_n$ be a function from $\mathcal H$
    attaining the minimum $\min_{h\in\mathcal H}
    \|h\circ\hat g_{n, \epsilon} - \Id_d\|_{\mathbb L_q}^q$.
    Since $\mathcal{F}$ contains all
    $1$-Lipschitz continuous functions, it also
    contains all the functions of the form
    $f = \psi \circ \hat{h}_n$, for $\psi \in \textup{Lip}(1/L_{\mathcal H})$. Using the
    notation $\bZ_i = \hat h_n(\bX_i)$, this
    implies that
    \begin{align}
        {\sf d}_{\mathcal{F}}(\widehat{g}_{n, \epsilon} \sharp\, \mathcal{U}_d, P_{n,X})
        &= \sup_{f \in \mathcal{F}} \bigg| \int_{[0,1]^d} f(\widehat{g}_{n, \epsilon}(\bx))d \bx - \frac{1}{n} \sum\limits_{i=1}^n f(\bX_i) \bigg| \\
        &\geqslant \sup_{\psi \in \text{Lip}_{L_{\mathcal H}^{-1}}} \bigg\{ \int_{[0,1]^d} (\psi\circ  \hat{h}_n\circ \widehat{g}_{n, \epsilon})(\bx)\,d\bx - \frac{1}{n} \sum\limits_{i=1}^n (\psi\circ \hat{h}_n)(\bX_i) \bigg\}\\
        &\geqslant \frac1{L_{\mathcal H}}
        \sup_{\psi \in \text{Lip}_1} \bigg\{ \int_{[0,1]^d} (\psi\circ \underbrace{ \hat{h}_n\circ \widehat{g}_{n, \epsilon}, }_{\approx\Id_d})(\bu)\,d\bu - \frac{1}{n} \sum\limits_{i=1}^n \psi(\bZ_i) \bigg\}.
        \label{eq:5}
    \end{align}
    By adding and subtracting the term $\int_{
    [0,1]^d} \psi(\bu)\,d\bu$, we arrive at
    \begin{align}
        {\sf d}_{\mathcal{F}}(\widehat{g}_{n, \epsilon} \sharp\, \mathcal{U}_d, P_{n,X})
        & \ge \frac1{L_{\mathcal{H}}}
        \sup_{\psi \in \text{Lip}_1}
        \bigg|\int_{[0,1]^d} \psi (\bu)\,d\bu - \frac{1}{n} \sum\limits_{i=1}^n \psi(\bZ_i) \bigg|
        - \frac1{L_{\mathcal{H}}}  \sup_{\psi \in \text{Lip}_1}
        \big\|\psi (\hat{h}_n \circ \widehat{g}_{n, \epsilon}
        ) - \psi\big\|_{\mathbb L_1}\\
        &\geqslant \frac1{L_{\mathcal{H}}} \Big(
        \wass_1(\mathcal{U}_d,\hat P_{n,Z})
        -
        \big\|\hat{h}_n \circ \widehat{g}_{n, \epsilon} -
        \Id_d\big\|_{\mathbb L_1}\Big)\\
        &\geq \frac1{L_{\mathcal{H}}} \Big(
        \wass_1(\mathcal{U}_d,\hat P_{n,Z})
        -   \textup{pen}_{\mathcal H}(\hat g_{n, \epsilon})^{1/q}\Big).
        \label{eq:decomposition1}
    \end{align}
    Here, the second inequality follows from the
    Lipschitz-continuity of $\psi$, whereas the
    last inequality is a consequence of the facts that
    the $\mathbb L_1$ norm on $[0,1]^d$ is dominated
    by the $\mathbb L_q$ norm given $q\ge 1$, and  that $\hat h_n$
    is a minimizer of $\|h\circ \hat g_n - \Id_d\|_q$
    over $\mathcal H$.
    We use the same lower bound here
    \begin{align}\label{eq:lower-bound-term1}
        \wass_1(\mathcal{U}_d,\hat P_{n,Z})
        \ge \frac{1}{2+2(2{\sf V}_dn)^{1/d}}.
    \end{align}
    To complete the proof we need to find an upper bound on the second term of the right-hand side of (\ref{eq:decomposition1}).
    Given (\ref{LIPERM}), we know that for any
    $g \in \mathcal{G}$ it holds
    \begin{align}
        {\sf d}_{\mathcal{F}}(\widehat{g}_{n, \epsilon}\sharp\, \mathcal{U}_d, P_{n,X}) + \lambda\, \textup{pen}_{\mathcal H}( \widehat{g}_{n, \epsilon}) &\le {\sf d}_{\mathcal{F}}(g\sharp\, \mathcal{U}_d, P_{n,X}) + \lambda\, \textup{pen}_{\mathcal H}(g) + \epsilon.
    \end{align}
    Since the distance ${\sf d}_{\mathcal F}$ is
    always nonnegative, by choosing $g$ from
    $\mathcal G_0$ the last term of the last
    display vanishes and we get
    \begin{align}\label{eq:penalty-bound}
        \lambda\,\textup{pen}_{\mathcal H}(\widehat{g}_{n, \epsilon})
        \leq  \inf_{g \in \mathcal{G}_0}
        {\sf d}_{\mathcal{F}}(g\sharp\, \mathcal{U}_d, P_{n,X}) + \epsilon.
    \end{align}
    Combining inequalities (\ref{eq:decomposition1}),
    (\ref{eq:lower-bound-term1}) and
    (\ref{eq:penalty-bound})
    we obtain the claim of the theorem.

\subsection{Proof of inequality (\ref{eq:6})}
\label{AppB:4}

Using the fact that $g^*\in \mathcal G_{\mathcal H}$ and
the triangle inequality, we get
\begin{align}
    \inf_{\mathcal{G}_{\mathcal H}} \mathbb E\big[
\wass_1(g\sharp\, \mathcal{U}_d, P_{n,X})\big]
&\leq \mathbb E\big[
\wass_1(g^*\sharp\, \mathcal{U}_d, P_{n,X})
\big]\\
&\leq \mathbb E\big[\wass_1(g^*\sharp\, \mathcal{U}_d, g^*\sharp P_{n,U})\big] + \mathbb E\big[\wass_1(g^*\sharp P_{n,U}, P_{n,X}) \big].\label{eq:6:1}
\end{align}
Let $\bU_i\sim \mathcal U_d$ be $n$ iid random vectors
drawn from the uniform distribution (they will be defined
more specifically later in the proof). Let $ P_{n,U}$ be
the empirical distribution of $\bU_1,\ldots,\bU_n$. Recall
that   (see, for example, \citep[Proposition 1]{weed2022estimation})
\begin{align}\label{wass1b}
    \mathbb E[\wass_1(\mathcal U_d,P_{n,U})] \leqslant \frac{c\sqrt{d}}{n^{1/d}}.
\end{align}
This implies that
\begin{align}
    \mathbb E\big[\wass_1(g^*\sharp\, \mathcal{U}_d, g^*\sharp P_{n,U})\big] & \leqslant L^*
    \mathbb E\big[\wass_1(\mathcal{U}_d,  P_{n,U})\big]
    \leqslant \frac{cL^*\sqrt{d}}{n^{1/d}}. \label{eq:6:2}
\end{align}
On the other hand, it is clear that
\begin{align}
    \mathbb E[\wass_1 ( g^*\sharp P_{n,U}, P_{n,X})] &\leq \mathbb E\bigg[(1/n)\sum_{i=1}^n \|g^*(\bU_i)  -\bX_i\|
    \bigg]\\
    & =  \mathbb E\big[\|g^*(\bU_1)  -\bX_1\|\big].\label{eq:6:3}
\end{align}
If we assume now that $\bU_1$ is chosen in such a way
that the joint distribution of $g^*(\bU_1)$ and $\bX_1$
is the optimal coupling between the marginal distributions
of these two random vectors, we get
\begin{align}
    \mathbb E\big[\|g^*(\bU_1)  -\bX_1\|\big] =
    \wass_1(g^*\sharp \mathcal U_d,P^*) \leqslant \sigma^*.
    \label{eq:6:4}
\end{align}
Combining (\ref{eq:6:3}) and (\ref{eq:6:4}), we get
\begin{align}
    \mathbb E[\wass_1 ( g^*\sharp P_{n,U}, P_{n,X})] &\leq
    \sigma^*.\label{eq:6:5}
\end{align}
Finally, inequalities (\ref{eq:6:2}) and (\ref{eq:6:5}),
in conjunction with (\ref{eq:6:1}), yield
\begin{align}
    \inf_{\mathcal{G}_{\mathcal H}} \mathbb E\big[
\wass_1(g\sharp\, \mathcal{U}_d, P_{n,X})\big]
&\leq \frac{cL^*\sqrt{d}}{n^{1/d}} + \sigma^*.
\label{eq}
\end{align}

\section{Proof of the risk bound when trained with approximated functional classes }

Although \Cref{prop:3} was stated for the case $q=2$
only, we provide the proof for any $q\geqslant 1$.
Replacing $q$ by $2$ in the final expression of
this proof leads to the claim of the proposition.

\subsection{Proof of \Cref{prop:3}}
We consider a trained generator $\widehat{g}_{n,\varepsilon}^0$ satisfying
\begin{align}\label{ineq:P2:1}
    {\sf d}_{\mathcal{F}_0} (\widehat{g}^0_{n,\varepsilon}\sharp\, U_d, P_{n,X}) &+ \textup{pen}_{\mathcal{H}_0} (\widehat{g}_{n,\varepsilon}^0)  \leq {\sf d}_{\mathcal{F}_0} (g \sharp\, U_d, P_{n,X}) + \textup{pen}_{\mathcal{H}_0} (g)
    + \varepsilon,\qquad
    \text{for all } g \in \mathcal{G}_0
\end{align}
and our goal is to upper bound the
expression ${\sf d}_{\mathcal{F}} (\widehat{g}^0_{n,\varepsilon}\sharp\, U_d, P_{n,X})$.
Using \Cref{lemma:1}, we have
\begin{align}
    {\sf d}_{\mathcal{F}} (\widehat{g}^0_{n,\varepsilon}\sharp\, U_d, P_{n,X})
    &\leqslant  {\sf d}_{\mathcal{F}_0} (\widehat{g}^0_{n,\varepsilon}\sharp\, U_d, P_{n,X}) + 2\delta_{\mathcal F}. \label{ineq:P2:2}
\end{align}
The first term of the
right-hand side can be upper bounded
as follows:
\begin{align}
    {\sf d}_{\mathcal{F}_0} (\widehat{g}^0_{n,\varepsilon}\sharp\, U_d, P_{n,X})
    &\leq {\sf d}_{\mathcal{F}_0} (\widehat{g}^0_{n,\varepsilon}\sharp\, U_d, P_{n,X}) + \lambda \textup{pen}_{\mathcal H_0}(\widehat{g}^0_{n,\varepsilon}) \\
    &\leq \inf_{g_0\in\mathcal G_0}\Big({\sf d}_{\mathcal{F}_0} (g_0\sharp\, U_d, P_{n,X}) + \lambda \textup{pen}_{\mathcal H_0}(g_0)\Big) + \varepsilon\\
    &\leq \inf_{g_0\in\mathcal G_0}\Big({\sf d}_{\mathcal{F}_0} (g_0\sharp\, U_d, P_{n,X}) + \lambda \textup{pen}_{\mathcal H}(g_0)\Big) + \varepsilon + q d^{(q-1)/2}\delta_{\mathcal H}\\
    &\leq \inf_{g_0\in\mathcal G_0}\Big({\sf d}_{\mathcal{F}} (g_0\sharp\, U_d, P_{n,X}) + \lambda \textup{pen}_{\mathcal H}(g_0)\Big) + \varepsilon + q d^{(q-1)/2}\delta_{\mathcal H}.\label{ineq:P2:3}
\end{align}
where we used the positiveness of the penalty function
for the first inequality, inequality (\ref{ineq:P2:1})
for the second inequality, \Cref{lemma:2} for the
third inequality and the fact that $\mathcal F_0\subseteq \mathcal F$ for the fourth inequality.

The last step is to use \Cref{lemma:3}, which allows to
upper bound the inf over $\mathcal G_0$ by the inf over
$\mathcal G$, modulo an additive error term proportional
to $\delta_{\mathcal G}$. More precisely,
\begin{align}
    \inf_{g_0\in\mathcal G_0}\Big({\sf d}_{\mathcal{F}} (g_0\sharp\, U_d, P_{n,X}) + \lambda \textup{pen}_{\mathcal H}(g_0)\Big)
    &\leqslant
    \inf_{g\in\mathcal G}\Big({\sf d}_{\mathcal{F}} (g\sharp\, U_d, P_{n,X}) + \lambda \textup{pen}_{\mathcal H}(g)\Big)  + \big(
    1+ \lambda q d^{(q-1)/2} L_{\mathcal H}\big)\delta_{\mathcal G}.
    \label{ineq:P2:4}
\end{align}
Combining (\ref{ineq:P2:2}), (\ref{ineq:P2:3}) and
(\ref{ineq:P2:4}), we get the inequality
\begin{align}
    {\sf d}_{\mathcal{F}} (\widehat{g}^0_{n,\varepsilon}\sharp\, U_d, P_{n,X})
    \leqslant \inf_{g\in\mathcal G}\Big({\sf d}_{\mathcal{F}} (g\sharp\, U_d, P_{n,X}) + \lambda \textup{pen}_{\mathcal H}(g)\Big)  + \eps + \underbrace{2\delta_{\mathcal F} +  q d^{(q-1)/2}\delta_{\mathcal H}+\big(
    1+ \lambda q d^{(q-1)/2} L_{\mathcal H}\big)\delta_{\mathcal G}}_{=:\delta}.
\end{align}
This completes the proof.

\subsection{Impact of approximating $\mathcal F$ on the IPM}

\begin{lemma}\label{lemma:1} If $\mathcal{F}_0$ is such that  $\inf\limits_{f_0\in\mathcal{F}_0}  \| f - f_0\| \leq \delta$ for every $f\in\mathcal{F}$, then
\begin{align}
    {\sf d}_\mathcal{F} (P,Q) - {\sf d}_{\mathcal{F}_0} (P,Q) \leq 2\delta \quad \textup{for all distributions} \quad P, Q.
\end{align}
\end{lemma}

\begin{proof} Recall the definition of ${\sf d}_\mathcal{F} (P, Q) = \sup_{f\in\mathcal{F}} \big|\mathbb{E}_P[f(X)]-\mathbb{E}_Q[f(X)]\big|$. This implies that
\begin{align}
    {\sf d}_{\mathcal{F}} (P,Q) - {\sf d}_{\mathcal{F}_0} (P,Q)  &
    =  \sup_{f\in \mathcal{F}} \inf_{f_0 \in \mathcal{F}_0} \Big(\underbrace{\big| \mathbb{E}_P[f(X)] - \mathbb{E}_Q[f(X)] \big|}_{\text{independent of $f_0$}} - \underbrace{\big| \mathbb{E}_P[f_0(X)] - \mathbb{E}_Q[f_0(X)] \big|}_{\text{independent of $f$}}
    \Big)\\
    &\leq \sup_{f\in \mathcal{F}} \inf_{f_0 \in \mathcal{F}_0} \big| \mathbb{E}_P[f(X)] - \mathbb{E}_Q[f(X)] - \mathbb{E}_P[f_0(X)] - \mathbb{E}_Q[f_0(X)] \big|\quad{\color{gray}(|a|-|b|\leqslant |a-b|)}\\
    &\leq \sup_{f\in\mathcal F}\inf_{f_0 \in \mathcal{F}_0} \Big(\big| \mathbb{E}_P[f(X)] - \mathbb{E}_P[f_0(X)]\big| + \big|\mathbb{E}_Q[f(X)]  - \mathbb{E}_Q[f_0(X)] \big|\Big) \qquad {\color{gray} (\text{triangle ineq.})}\\
    &\leq \sup_{f\in\mathcal F}\inf_{f_0 \in \mathcal{F}_0} \Big( \mathbb{E}_P[\underbrace{|f(X) - f_0(X)|}_{\leqslant \|f-f_0\|_\infty}] + \mathbb{E}_Q[\underbrace{|f(X) - f_0(X)|}_{\leqslant \|f-f_0\|_\infty}] \Big)\\
    &\leq \sup_{f\in\mathcal F}\inf_{f_0 \in \mathcal{F}_0}
    \|f-f_0\|_\infty \leq 2\delta.
\end{align}
This completes the proof of the lemma.
\end{proof}

\subsection{Impact of approximating $\mathcal H$ on the Left-Inverse-Penalty}

Before analyzing the sensitivity of the left-inverse-penalty to the deviations from $\mathcal H$, we need
an auxiliary lemma.
\begin{lemma}\label{prop:aq_bq}
If $a,b$ are arbitrary numbers from some interval $[0,C]$,
and $q\geqslant 1$, then
\begin{align}
    |a^q - b^q| \leq  qC^{q-1}|b-a|.
\end{align}
\end{lemma}
\begin{proof}
Let us first assume that $c\in [0,1]$. For any $q\geqslant 1$,
\begin{align}
    |c^q - 1| \leqslant q|c-1|.
\end{align}
Then for $a,b\in \mathbb R$ such that $0\leqslant a\leqslant b$, we have
\begin{align}
    |a^q - b^q| = b^q\big|(a/b)^q - 1 \big| \leq
    q b^q\Big|\frac{a}{b} - 1 \Big| = qb^{q-1}|b-a|.
\end{align}
The claim of the lemma follows by upper bounding
$b$ by $C$.
\end{proof}

\begin{lemma}\label{lemma:2} If
$\mathcal H_0$ is such that $\min\limits_{h_0\in\mathcal{H}_0} \|h-h_0\|_\infty
\leq \delta$ for all $h \in \mathcal{H}$, then
\begin{align}
    \textup{pen}_{\mathcal{H}_0} (g) - \textup{pen}_{\mathcal{H}} (g)\leq qd^{(q-1)/2}\delta,\qquad
    \text{for all }g\in\mathcal G.
\end{align}
\end{lemma}

\begin{proof}
Recall that
$\textup{pen}_{\mathcal{H}} (g) = \min_{h\in \mathcal{H}} \|h\circ g - \Id_d\|_{\mathbb L_q}^q $. This yields
\begin{align}
    \textup{pen}_{\mathcal{H}_0} (g) - \textup{pen}_{\mathcal{H}} (g)
    &=  \min_{h_0\in \mathcal{H}_0} \| h_0\circ g - \Id_d \|_{\mathbb L_q}^q - \min_{h\in \mathcal{H}} \| h\circ g - \Id_d \|_{\mathbb L_q}^q  \\[5pt]
    &=  \max_{h\in \mathcal{H}} \min_{h_0\in \mathcal{H}_0} \Big(\| h_0\circ g - \Id_d \|_{\mathbb L_q}^q - \| h\circ g - \Id_d \|_{\mathbb L_q}^q \Big)
\end{align}
We apply \Cref{prop:aq_bq} with
\begin{align}
    a &=  \| h_0\circ g - \Id_d \|_{\mathbb L_q}
    =\bigg(\int_{[0,1]^d} \big\|\underbrace{h_0(g(u))}_{\in [0,1]^d} - \underbrace{u}_{\in[0,1]^d}\|^q\,\rmd u\bigg)^{1/q}\leqslant \sqrt{d}\\
    b &=  \| h\circ g - \Id_d \|_{\mathbb L_q}\leqslant \sqrt{d}.
\end{align}
This leads to
\begin{align}
    \| h_0\circ g - \Id_d \|_{\mathbb L_q}^q
    - \| h\circ g - \Id_d \|_{\mathbb L_q}^q
    &\leqslant q d^{(q-1)/2} \Big| \| h_0\circ g - \Id_d \|_{\mathbb L_q}
    - \| h\circ g - \Id_d \|_{\mathbb L_q}\Big|
    \quad{\color{gray}(|\|a\|-\|b\||\leqslant \|a-b\|)}\\
    &\leqslant q d^{(q-1)/2} \| h_0\circ g - \Id_d - h\circ g + \Id_d \|_{\mathbb L_q}\\
    &\leqslant q d^{(q-1)/2} \| h_0\circ g  - h\circ g  \|_{\mathbb L_q}\\
    &= q d^{(q-1)/2} \| h_0  - h \|_{\infty}\leqslant q d^{(q-1)/2} \delta.
\end{align}
This completes the proof of the lemma
\end{proof}

\subsection{Impact of approximating $\mathcal G$ on
the Penalized Empirical Risk}

As in the main text, here also we assume that
the elements of the set $\mathcal H$ are all
Lipschitz-continuous with a Lipschitz constant
bounded by $L_{\mathcal H}$.

\begin{lemma}\label{lemma:3}
    Let $P$ be an arbitrary distribution and
    $\mathcal F\subseteq \textup{Lip}_1(\mathbb R^D\to\mathbb R)$.
    If $\mathcal G_0$ is such that $\min_{g_0\in \mathcal G_0} \|g-g_0\|_\infty \leqslant \delta$ for every
    $g\in\mathcal G$, then the following is true:
    \begin{align}
        \min_{g_0\in \mathcal G_0}\Big({\sf d}_{\mathcal{F}} (g_0\sharp\, U_d, P)
        + \lambda \textup{pen}_{\mathcal H}(g_0)\Big)
        \leqslant \inf_{g\in \mathcal G}
        {\sf d}_{\mathcal{F}} (g\sharp\, U_d, P)
        + \lambda \textup{pen}_{\mathcal H}(g) +
        \big( 1 + \lambda q d^{(q-1)/2} L_{\mathcal H}\big)\delta.
    \end{align}
\end{lemma}
\begin{proof}
    Let $g$ be any function from $\mathcal G$ and
    $g_0$ be the element of $\mathcal G_0$ satisfying
    $\|g-g_0\|_\infty \leqslant \delta$. On the one hand, for this function $g_0$, we have
    \begin{align}
        {\sf d}_{\mathcal{F}} (g_0\sharp\, U_d, P) -
        {\sf d}_{\mathcal{F}} (g\sharp\, U_d, P) &=
        \sup_{f\in\mathcal F}\Big| \mathbb E_{\mathcal U_d}[f(g_0(U))] - \mathbb E_P[f(X)]\Big| - \sup_{f\in\mathcal F}\Big| \mathbb E_{\mathcal U_d}[f(g(U))] - \mathbb E_P[f(X)]\Big|\\
        &\leqslant \sup_{f\in\mathcal F}\Big| \mathbb E_{\mathcal U_d}\Big[f(g_0(U)) - f(g(U))\Big]\Big| \\
        &\leqslant \sup_{f\in\mathcal F} \mathbb E_{\mathcal U_d}\Big[\Big|f(g_0(U)) - f(g(U))\Big|\Big] \\
        &\leqslant \sup_{f\in\mathcal F} \mathbb E_{\mathcal U_d}\Big[\big\|g_0(U) - g(U)\big\|\Big] \leqslant \delta.\label{L4:ineq2}
    \end{align}
    Here, the inequality of the second line follows from
    $\sup |F|- \sup |G| \leqslant \sup (|F|-|G|) \leqslant
    \sup |F-G|$, while the first inequality of the last line
    is a consequence of the assumption that the functions
    from $\mathcal F$ are all $1$-Lipschitz.

    On the other hand,
    \begin{align}
        \textup{pen}_{\mathcal H} (g_0) -
        \textup{pen}_{\mathcal H} (g) &=
        \min_{h\in\mathcal H} \|h\circ g_0 - \Id_d\|_{\mathbb L_q}^q - \min_{h\in\mathcal H} \|h\circ g - \Id_d\|_{\mathbb L_q}^q\\
        &\leq \|h^*\circ g_0 - \Id_d\|_{\mathbb L_q}^q -  \|h^*\circ g - \Id_d\|_{\mathbb L_q}^q,
    \end{align}
    where $h^*$ is the minimizer of $\|h\circ g -
    \Id_d\|_{\mathbb L_q}^q$ over $\mathcal H$. Combining with \Cref{prop:aq_bq}, we get
    \begin{align}
        \textup{pen}_{\mathcal H} (g_0) -
        \textup{pen}_{\mathcal H} (g) &\leqslant
        q d^{(q-1)/2} \Big|\|h^*\circ g_0 - \Id_d\|_{\mathbb L_q} - \|h^*\circ g - \Id_d\|_{\mathbb L_q}\Big|\\
        &\leqslant q d^{(q-1)/2} \|h^*\circ g_0 - \Id_d - h^*\circ g + \Id_d\|_{\mathbb L_q}\\
        &= q d^{(q-1)/2} \|h^*\circ g_0 - h^*\circ g \|_{\mathbb L_q}\\
        &\leqslant q d^{(q-1)/2} \|h^*\circ g_0 - h^*\circ g \|_\infty\\
        &\leqslant q d^{(q-1)/2} L_{\mathcal H}
        \| g_0 - g \|_\infty\\
        &\leqslant q d^{(q-1)/2} L_{\mathcal H} \delta.\label{L4:ineq3}
    \end{align}
    Combining (\ref{L4:ineq2}) and (\ref{L4:ineq3}), we get
    that for any $g\in\mathcal G$, there is a $g_0\in\mathcal G_0$ such that
    \begin{align}
        {\sf d}_{\mathcal{F}} (g_0\sharp\, U_d, P)
        + \lambda \textup{pen}_{\mathcal H}(g_0)
        \leqslant
        {\sf d}_{\mathcal{F}} (g\sharp\, U_d, P)
        + \lambda \textup{pen}_{\mathcal H}(g) +
        \big( 1 + \lambda q d^{(q-1)/2} L_{\mathcal H}\big)\delta.
    \end{align}
    This completes the proof of the lemma.
\end{proof}

\section{Numerical experiments}

This section contains some  tables and plots that we could not
include in the main paper due to the space restrictions. Recall
that our experiments were done on three data sets: Swiss roll, MNIST and CIFAR-10. In these experiments, we trained a distribution using the Wasserstein GAN with a left inverse penalty.

\begin{figure}[htbp]
    \centering

    \begin{subfigure}{0.24\textwidth}
        \includegraphics[width=\linewidth]{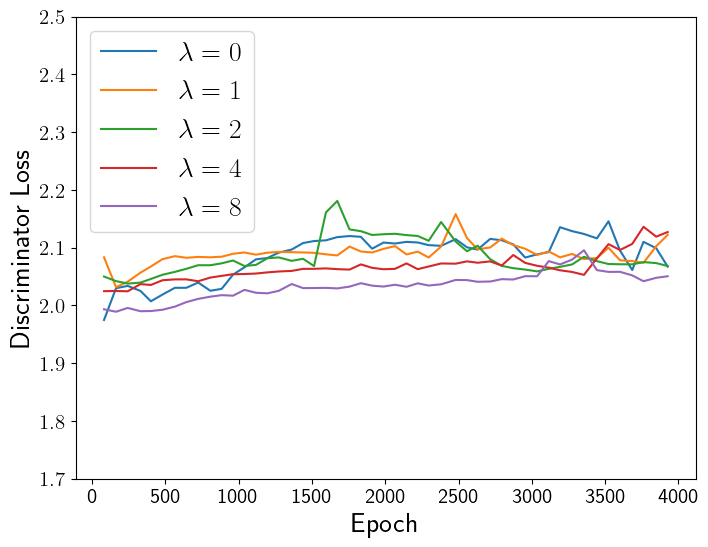}
        \caption{Discriminator Loss}
        \label{subfig:a}
    \end{subfigure}
    \hfill
    \begin{subfigure}{0.24\textwidth}
        \includegraphics[width=\linewidth]{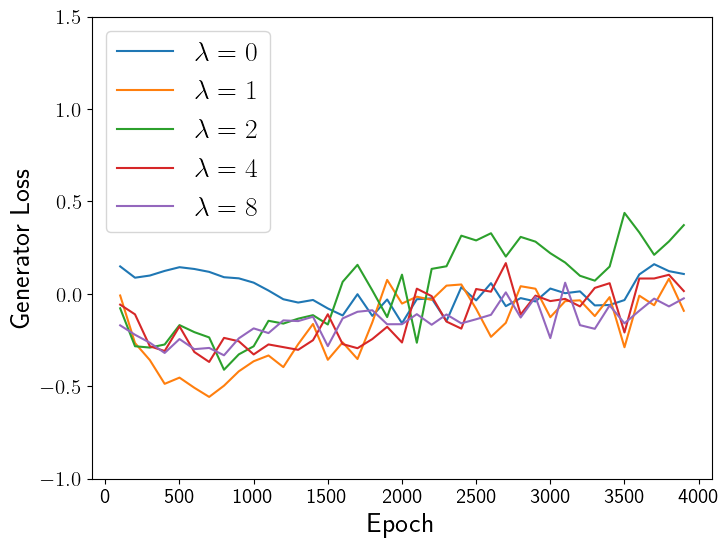}
        \caption{Generator Loss}
        \label{subfig:b}
    \end{subfigure}
    \hfill
    \begin{subfigure}{0.24\textwidth}
        \includegraphics[width=\linewidth]{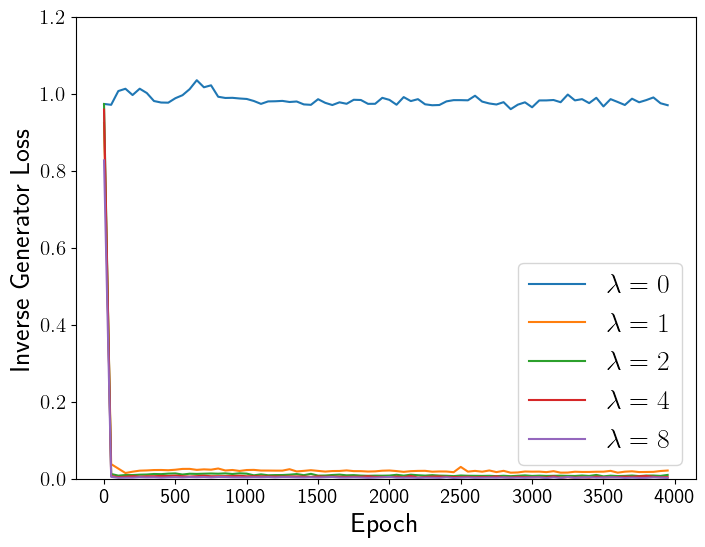}
        \caption{Inverse Generator Loss}
        \label{subfig:c}
    \end{subfigure}
    \hfill
    \begin{subfigure}{0.24\textwidth}
        \includegraphics[width=\linewidth]{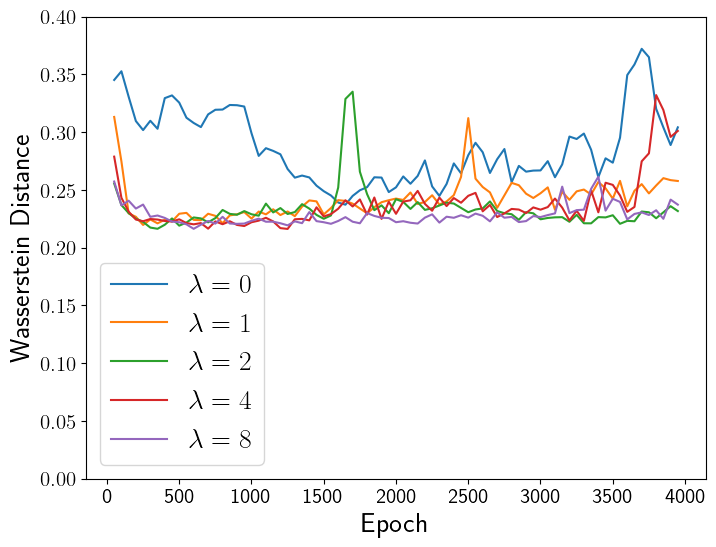}
        \caption{Wasserstein distance}
        \label{subfig:c}
    \end{subfigure}
    \vspace*{-10pt}
    \caption{Evolution of various losses across the iterations in the experiment on Swiss Roll data generated with the
    noise magnitude $\sigma = 1.5$.}
    \label{fig:swiss_roll_gen2}
\end{figure}

\begin{figure}[htbp]
    \centering
    \begin{subfigure}{0.18\textwidth}
        \includegraphics[width=\linewidth]{newest_lambda_0.png}
        \caption{$\lambda=0$. }
        \label{subfig:a}
    \end{subfigure}
    \hfill
    \begin{subfigure}{0.18\textwidth}
        \includegraphics[width=\linewidth]{newest_lambda_1.png}
        \caption{$\lambda=1$.}
        \label{subfig:b}
    \end{subfigure}
    \hfill
    \begin{subfigure}{0.18\textwidth}
        \includegraphics[width=\linewidth]{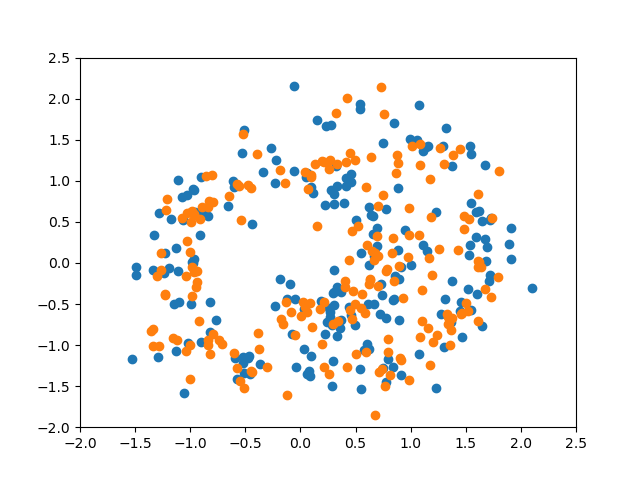}
        \caption{$\lambda=2$.}
        \label{subfig:c}
    \end{subfigure}
    \hfill
    \begin{subfigure}{0.18\textwidth}
        \includegraphics[width=\linewidth]{newest_lambda_4.png}
        \caption{$\lambda=4$.}
        \label{subfig:c}
    \end{subfigure}
    \hfill
    \begin{subfigure}{0.18\textwidth}
        \includegraphics[width=\linewidth]{newest_lambda_8.png}
        \caption{$\lambda=8$.}
        \label{subfig:d}
    \end{subfigure}

    \begin{subfigure}{0.18\textwidth}
        \includegraphics[width=\linewidth]{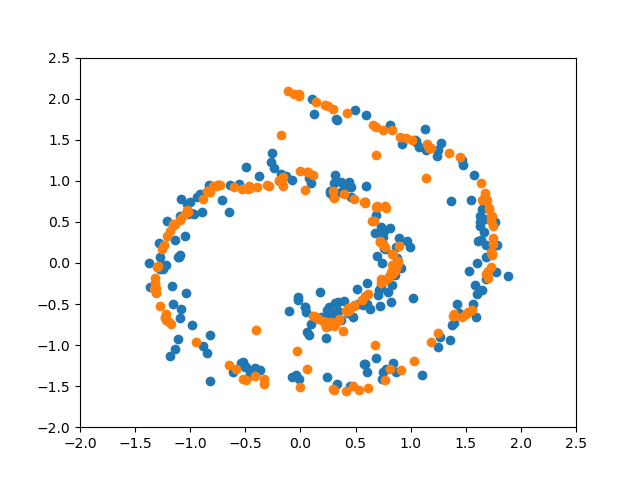}
        \caption{$\lambda=0$. }
        \label{subfig:a}
    \end{subfigure}
    \hfill
    \begin{subfigure}{0.18\textwidth}
        \includegraphics[width=\linewidth]{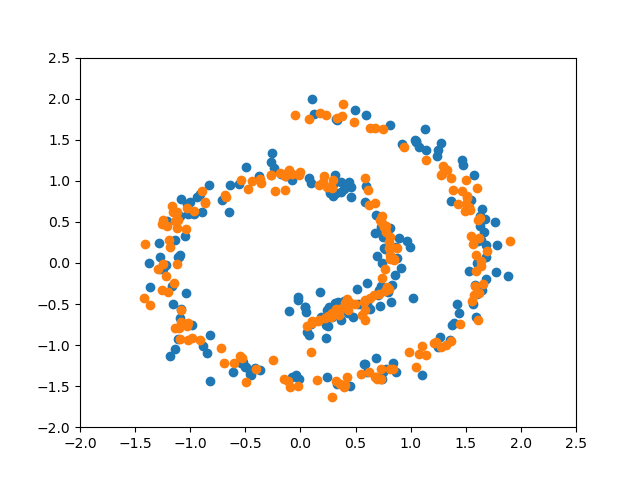}
        \caption{$\lambda=1$.}
        \label{subfig:b}
    \end{subfigure}
    \hfill
    \begin{subfigure}{0.18\textwidth}
        \includegraphics[width=\linewidth]{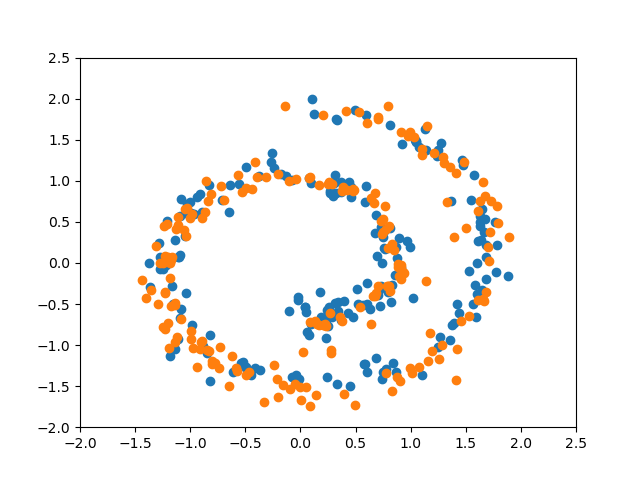}
        \caption{$\lambda=2$.}
        \label{subfig:c}
    \end{subfigure}
    \hfill
    \begin{subfigure}{0.18\textwidth}
        \includegraphics[width=\linewidth]{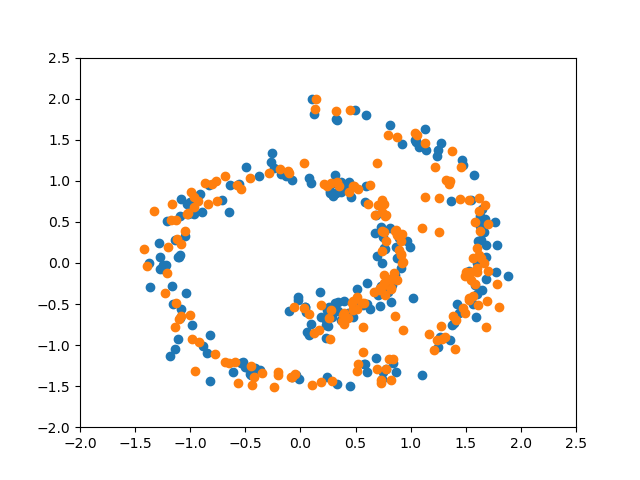}
        \caption{$\lambda=4$.}
        \label{subfig:c}
    \end{subfigure}
    \hfill
    \begin{subfigure}{0.18\textwidth}
        \includegraphics[width=\linewidth]{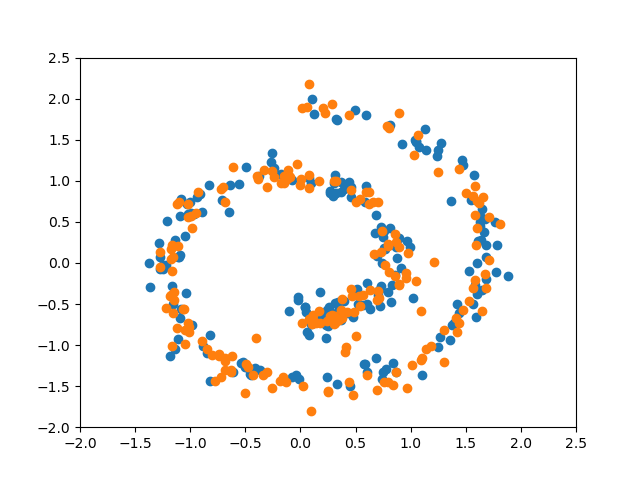}
        \caption{$\lambda=8$.}
        \label{subfig:d}
    \end{subfigure}

    \begin{subfigure}{0.18\textwidth}
        \includegraphics[width=\linewidth]{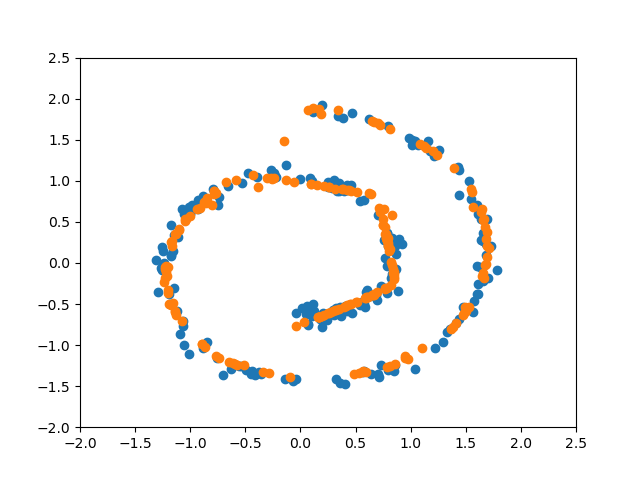}
        \caption{$\lambda=0$. }
        \label{subfig:a}
    \end{subfigure}
    \hfill
    \begin{subfigure}{0.18\textwidth}
        \includegraphics[width=\linewidth]{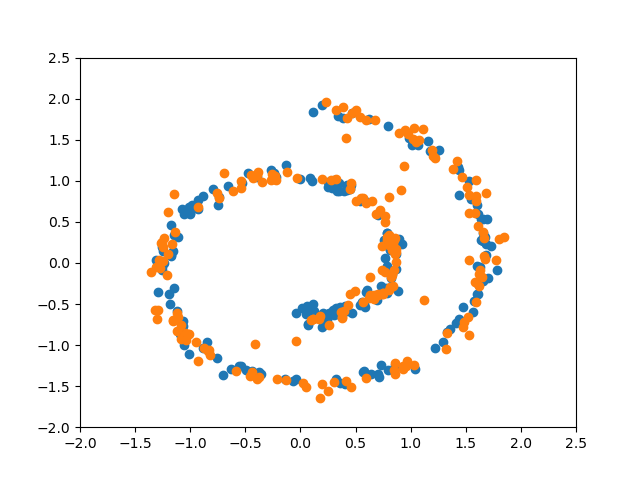}
        \caption{$\lambda=1$.}
        \label{subfig:b}
    \end{subfigure}
    \hfill
    \begin{subfigure}{0.18\textwidth}
        \includegraphics[width=\linewidth]{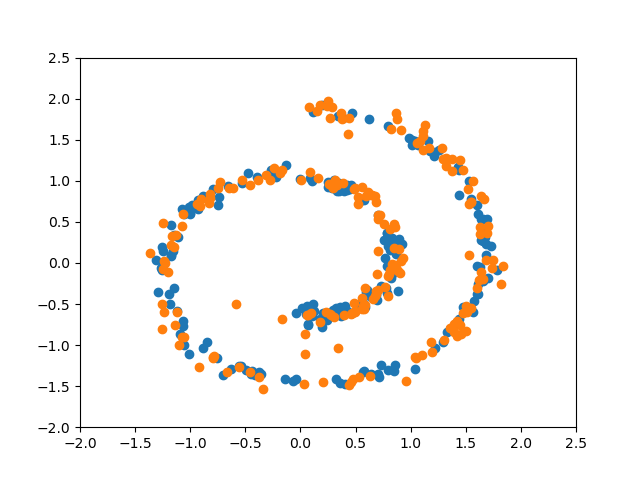}
        \caption{$\lambda=2$.}
        \label{subfig:c}
    \end{subfigure}
    \hfill
    \begin{subfigure}{0.18\textwidth}
        \includegraphics[width=\linewidth]{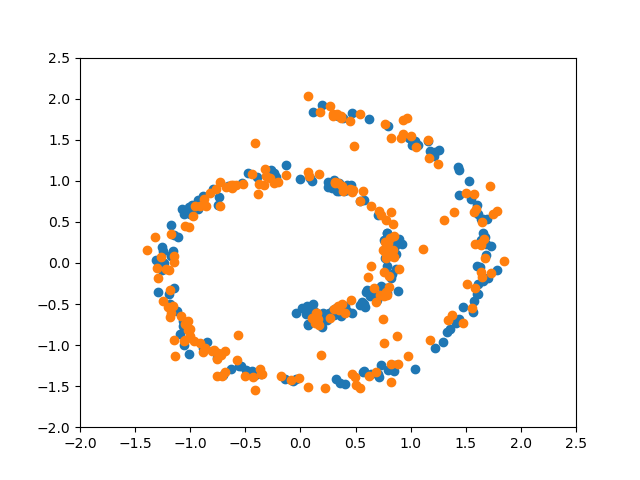}
        \caption{$\lambda=4$.}
        \label{subfig:c}
    \end{subfigure}
    \hfill
    \begin{subfigure}{0.18\textwidth}
        \includegraphics[width=\linewidth]{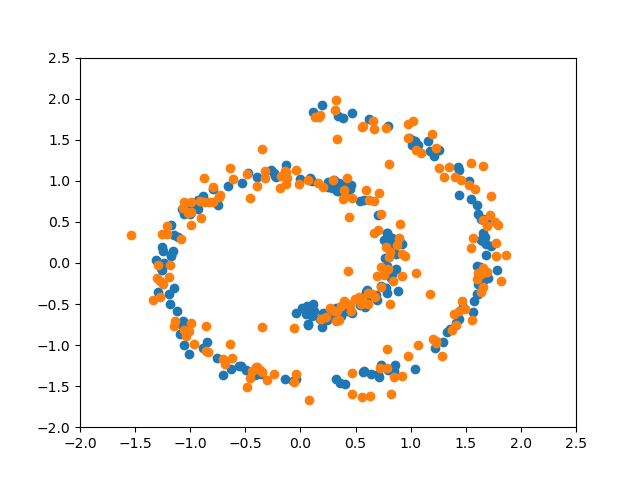}
        \caption{$\lambda=8$.}
        \label{subfig:d}
    \end{subfigure}

    \caption{Swiss Roll data: training set
    and examples generated from the learned distribution,
    when the noise magnitude is $\sigma = 3/2$ (top row),
    $\sigma= 3/4$ middle row and $\sigma = 3/8$ (bottom row). LIPERM was trained with $\lambda = 0,1,2,4,8$.}
    \label{fig:swiss_large_sigma}
\end{figure}

\begin{table}[h!]
    \begin{minipage}{0.65\textwidth}
        \centering
        \newcommand{\inputnum}{2}
        \newcommand{\hiddennum}{5}
        \newcommand{\outputnum}{2}
        \begin{tcolorbox}
        \hspace*{-15pt}
        \begin{tikzpicture}[scale=0.9,font=\footnotesize]
        \foreach \i in {1,...,\inputnum}
        {
            \node[circle,
                minimum size = 6mm,
                fill=orange!30] (Input-\i) at (0,-\i) {\tiny $U^\i$};
        }

        \foreach \i in {1,...,\hiddennum}
        {
            \node[circle,
                minimum size = 6mm,
                fill=teal!50,
                yshift=(\hiddennum-\inputnum)*5 mm
            ] (Hidden1-\i) at (2.5,-\i) {\tiny ReLU};
        }

        \foreach \i in {1,...,\hiddennum}
        {
            \node[circle,
                minimum size = 6mm,
                fill=teal!50,
                yshift=(\hiddennum-\inputnum)*5 mm
            ] (Hidden2-\i) at (5,-\i) {\tiny ReLU};
        }

        \foreach \i in {1,...,\hiddennum}
        {
            \node[circle,
                minimum size = 6mm,
                fill=teal!50,
                yshift=(\hiddennum-\inputnum)*5 mm
            ] (Hidden3-\i) at (7.5,-\i) {\tiny ReLU};
        }

        \foreach \i in {1,...,\outputnum}
        {
            \node[circle,
                minimum size = 6mm,
                fill=purple!50,
                yshift=(\outputnum-\inputnum)*5 mm
            ] (Output-\i) at (10,-\i) {\tiny$X^\i$};
        }

        \foreach \i in {1,...,\inputnum}
        {
            \foreach \j in {1,...,\hiddennum}
            {
                \draw[->, shorten >=1pt] (Input-\i) -- (Hidden1-\j);
            }
        }

        \foreach \i in {1,...,\hiddennum}
        {
            \foreach \j in {1,...,\hiddennum}
            {
                \draw[->, shorten >=1pt] (Hidden1-\i) -- (Hidden2-\j);
            }
        }

        \foreach \i in {1,...,\hiddennum}
        {
            \foreach \j in {1,...,\hiddennum}
            {
                \draw[->, shorten >=1pt] (Hidden2-\i) -- (Hidden3-\j);
            }
        }

        \foreach \i in {1,...,\hiddennum}
        {
            \foreach \j in {1,...,\outputnum}
            {
                \draw[->, shorten >=1pt] (Hidden3-\i) -- (Output-\j);
            }
        }

        \node [below=1cm of Input-2, align=center] {
        Input layer \\ dim = 2};
        \node [below=of Hidden1-4, align=center] {1st hid.\ layer \\ dim = 512};
        \node [below=of Hidden2-4, align=center] {2nd hid.\ layer\\ dim = 512};
        \node [below=of Hidden3-4, align=center] {3rd hid.\ layer\\ dim = 512};
        \node [below=1cm of Output-2, align=center] {Output layer\\ dim = 2};

        \node [below=2cm of Hidden2-4, align=center] {\sc\color{blue}\large The generator network};
        \end{tikzpicture}
        \end{tcolorbox}
    \end{minipage}
    \hspace*{10pt}
    \begin{minipage}{0.35\textwidth}
    \begin{tabular}{c|c}
        \toprule
        \textbf{Layer} & \textbf{Operation} \\
        \midrule
        \multirow{ 2}{*}{1} & Linear + ReLU \\
         & $\text{latent\_DIM} \rightarrow \text{DIM}$ \\
        \midrule
        \multirow{ 2}{*}{2} & Linear + ReLU \\
        & $\text{DIM} \rightarrow \text{DIM}$ \\
        \midrule
        \multirow{ 2}{*}{3} & Linear + ReLU \\
        & $\text{DIM} \rightarrow \text{DIM}$ \\
        \midrule
        \multirow{ 2}{*}{Out} & Linear \\
         & $\text{DIM} \rightarrow \text{out\_DIM}$\\
        \bottomrule
    \end{tabular}
    \end{minipage}%

    \bigskip
    \begin{minipage}{0.65\textwidth}
        \centering
        \newcommand{\inputnum}{2}
        \newcommand{\hiddennum}{5}
        \newcommand{\outputnum}{1}
        \begin{tcolorbox}
        \hspace*{-15pt}
        \begin{tikzpicture}[scale=0.9,font=\footnotesize]
        \foreach \i in {1,...,\inputnum}
        {
            \node[circle,
                minimum size = 6mm,
                fill=orange!30] (Input-\i) at (0,-\i) {\tiny $X^\i$};
        }

        \foreach \i in {1,...,\hiddennum}
        {
            \node[circle,
                minimum size = 6mm,
                fill=teal!50,
                yshift=(\hiddennum-\inputnum)*5 mm
            ] (Hidden1-\i) at (2.5,-\i) {\tiny ReLU};
        }

        \foreach \i in {1,...,\hiddennum}
        {
            \node[circle,
                minimum size = 6mm,
                fill=teal!50,
                yshift=(\hiddennum-\inputnum)*5 mm
            ] (Hidden2-\i) at (5,-\i) {\tiny ReLU};
        }

        \foreach \i in {1,...,\hiddennum}
        {
            \node[circle,
                minimum size = 6mm,
                fill=teal!50,
                yshift=(\hiddennum-\inputnum)*5 mm
            ] (Hidden3-\i) at (7.5,-\i) {\tiny ReLU};
        }

        \foreach \i in {1,...,\outputnum}
        {
            \node[circle,
                minimum size = 6mm,
                fill=purple!50,
                yshift=(\outputnum-\inputnum)*5 mm
            ] (Output-\i) at (10,-\i) {\tiny$Y$};
        }

        \foreach \i in {1,...,\inputnum}
        {
            \foreach \j in {1,...,\hiddennum}
            {
                \draw[->, shorten >=1pt] (Input-\i) -- (Hidden1-\j);
            }
        }

        \foreach \i in {1,...,\hiddennum}
        {
            \foreach \j in {1,...,\hiddennum}
            {
                \draw[->, shorten >=1pt] (Hidden1-\i) -- (Hidden2-\j);
            }
        }

        \foreach \i in {1,...,\hiddennum}
        {
            \foreach \j in {1,...,\hiddennum}
            {
                \draw[->, shorten >=1pt] (Hidden2-\i) -- (Hidden3-\j);
            }
        }

        \foreach \i in {1,...,\hiddennum}
        {
            \foreach \j in {1,...,\outputnum}
            {
                \draw[->, shorten >=1pt] (Hidden3-\i) -- (Output-\j);
            }
        }

        \node [below=1cm of Input-2, align=center] {
        Input layer \\ dim = 2};
        \node [below=of Hidden1-4, align=center] {1st hid.\ layer \\ dim = 512};
        \node [below=of Hidden2-4, align=center] {2nd hid.\ layer\\ dim = 512};
        \node [below=of Hidden3-4, align=center] {3rd hid.\ layer\\ dim = 512};
        \node [below=1cm of Output-2, align=center] {Output layer\\ dim = 1};

        \node [below=2cm of Hidden2-4, align=center] {\sc\color{blue}\large The critics network};
        \end{tikzpicture}
        \end{tcolorbox}
    \end{minipage}
    \hspace*{10pt}
    \begin{minipage}{0.35\textwidth}
    \begin{tabular}{c|c}
        \toprule
        \textbf{Layer} & \textbf{Operation} \\
        \midrule
        \multirow{ 2}{*}{1} & Linear + ReLU \\
         & $\text{out\_DIM} \rightarrow \text{DIM}$ \\
        \midrule
        \multirow{ 2}{*}{2} & Linear + ReLU \\
        & $\text{DIM} \rightarrow \text{DIM}$ \\
        \midrule
        \multirow{ 2}{*}{3} & Linear + ReLU \\
        & $\text{DIM} \rightarrow \text{DIM}$ \\
        \midrule
        \multirow{ 2}{*}{Out} & Linear \\
         & $\text{DIM} \rightarrow 1$\\
        \bottomrule
    \end{tabular}
    \end{minipage}%

    \bigskip

    \begin{minipage}{0.65\textwidth}
        \centering
        \newcommand{\inputnum}{2}
        \newcommand{\hiddennuma}{5}
        \newcommand{\hiddennumb}{5}
        \newcommand{\outputnum}{2}
        \begin{tcolorbox}
        \hspace*{-15pt}
        \begin{tikzpicture}[scale=0.9,font=\footnotesize]
        \foreach \i in {1,...,\inputnum}
        {
            \node[circle,
                minimum size = 6mm,
                fill=orange!30] (Input-\i) at (0,-\i) {\tiny $X^\i$};
        }

        \foreach \i in {1,...,\hiddennuma}
        {
            \node[circle,
                minimum size = 6mm,
                fill=teal!50,
                yshift=(\hiddennuma-\inputnum)*5 mm
            ] (Hidden1-\i) at (2.5,-\i) {\tiny ReLU};
        }

        \foreach \i in {1,...,\hiddennumb}
        {
            \node[circle,
                minimum size = 6mm,
                fill=teal!50,
                yshift=(\hiddennumb-\inputnum)*5 mm
            ] (Hidden2-\i) at (5,-\i) {\tiny ReLU};
        }

        \foreach \i in {1,...,\hiddennumb}
        {
            \node[circle,
                minimum size = 6mm,
                fill=teal!50,
                yshift=(\hiddennumb-\inputnum)*5 mm
            ] (Hidden3-\i) at (7.5,-\i) {\tiny ReLU};
        }

        \foreach \i in {1,...,\outputnum}
        {
            \node[circle,
                minimum size = 6mm,
                fill=purple!50,
                yshift=(\outputnum-\inputnum)*5 mm
            ] (Output-\i) at (10,-\i) {\tiny$U^\i$};
        }

        \foreach \i in {1,...,\inputnum}
        {
            \foreach \j in {1,...,\hiddennuma}
            {
                \draw[->, shorten >=1pt] (Input-\i) -- (Hidden1-\j);
            }
        }

        \foreach \i in {1,...,\hiddennuma}
        {
            \foreach \j in {1,...,\hiddennumb}
            {
                \draw[->, shorten >=1pt] (Hidden1-\i) -- (Hidden2-\j);
            }
        }

        \foreach \i in {1,...,\hiddennumb}
        {
            \foreach \j in {1,...,\hiddennumb}
            {
                \draw[->, shorten >=1pt] (Hidden2-\i) -- (Hidden3-\j);
            }
        }

        \foreach \i in {1,...,\hiddennumb}
        {
            \foreach \j in {1,...,\outputnum}
            {
                \draw[->, shorten >=1pt] (Hidden3-\i) -- (Output-\j);
            }
        }

        \node [below=1cm of Input-2, align=center] {
        Input layer \\ dim = 2};
        \node [below=of Hidden1-4, align=center] {1st hid.\ layer \\ dim = 512};
        \node [below=of Hidden2-4, align=center] {2nd hid.\ layer\\ dim = 512};
        \node [below=of Hidden3-4, align=center] {3rd hid.\ layer\\ dim = 512};
        \node [below=1cm of Output-2, align=center] {Output layer\\ dim = 2};

        \node [below=2.5cm of Hidden2-4, align=center] {\sc\color{blue}\large The left inverse network};
        \end{tikzpicture}
    \end{tcolorbox}
    \end{minipage}
    \hspace{10pt}
    \begin{minipage}{0.35\textwidth}
    \begin{tabular}{c|c}
        \toprule
        \textbf{Layer} & \textbf{Operation} \\
        \midrule
        \multirow{ 2}{*}{1} & Linear + LeakyReLU \\
         & $\text{out\_DIM} \rightarrow \text{DIM}$ \\
        \midrule
        \multirow{ 2}{*}{2} & Linear + LeakyReLU \\
        & $\text{DIM} \rightarrow \text{DIM}$ \\
        \midrule
        \multirow{ 2}{*}{3} & Linear + LeakyReLU \\
        & $\text{DIM} \rightarrow \text{DIM}$ \\
        \midrule
        \multirow{ 2}{*}{Out} & Linear \\
         & $\text{DIM} \rightarrow \text{latent\_DIM}$\\
        \bottomrule
    \end{tabular}
    \end{minipage}%

    \caption{Neural network architectures for the generator $g$, the critic $f$ and the left inverse $h$ used in the experiments conducted on Swiss Roll datasets. In this case, \text{latent\_DIM} = \text{out\_DIM} = 2, DIM = 512.}
    \label{tab:discriminator_toy}
\end{table}

\begin{figure}[h!]
        \centering
        \begin{subfigure}[b]{0.48\textwidth}
            \centering
            \includegraphics[width=0.98\textwidth]{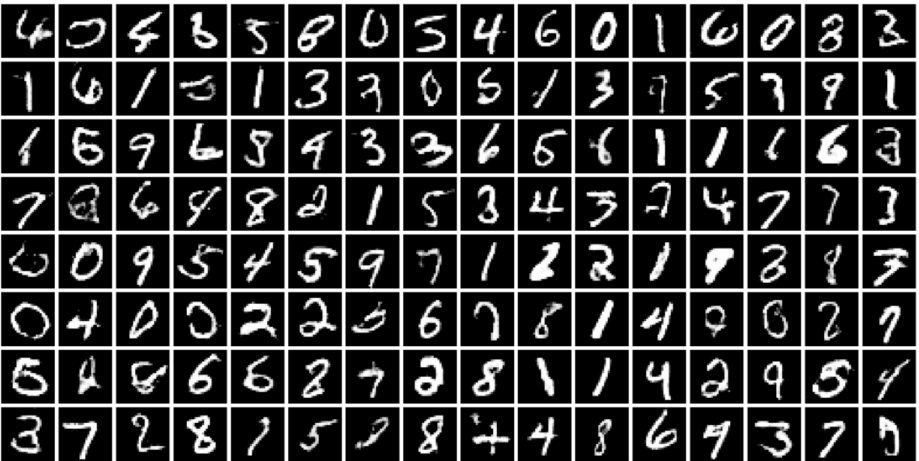}
            \caption[ $\lambda = 0$.]%
            {{\small  $\lambda = 0$.}}
            \label{fig:lambda0}
        \end{subfigure}
        \begin{subfigure}[b]{0.48\textwidth}
            \centering
            \includegraphics[width=0.98\textwidth]{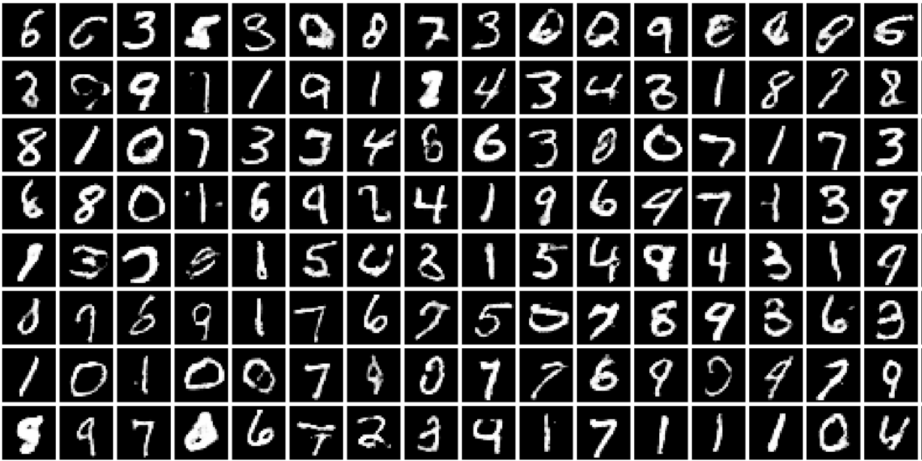}
            \caption[ $\lambda = 1$.]%
            {{\small  $\lambda = 1$.}}
            \label{fig:lambda1}
        \end{subfigure}

        \vskip\baselineskip

        \begin{subfigure}[b]{0.48\textwidth}
            \centering
            \includegraphics[width=0.98\textwidth]{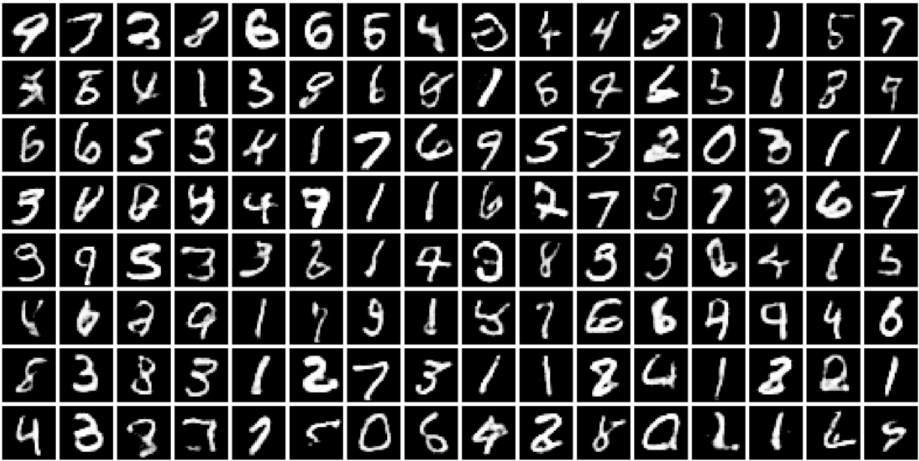}
            \caption[ $\lambda = 4$.]%
            {{\small $\lambda = 4$.}}
            \label{fig:lambda4}
        \end{subfigure}
        \begin{subfigure}[b]{0.48\textwidth}
            \centering 
            \includegraphics[width=0.98\textwidth]{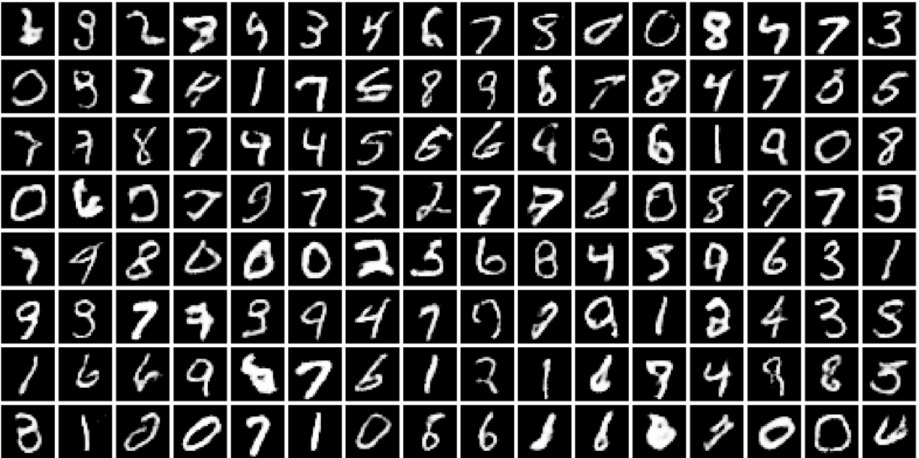}
            \caption[ $\lambda = 8$.]%
            {{\small $\lambda = 8$.}}
            \label{fig:lambda8}
        \end{subfigure}
        \vskip\baselineskip

        \caption[ Handwritten digits generated for different values of $\lambda$.]
        {\small Handwritten digits generated by LIPERM for different values of $\lambda$.}
        \label{fig:mnist2}
\end{figure}

\begin{table}[htbp]
  \centering
  \begin{subtable}[b]{\textwidth}
    \centering
      \caption{Discriminator Network Architecture used in our experiments on MNIST data set. }
\vspace{10pt}
 \begin{tabular}{l|c|c|c|c}
    \toprule

      & Conv2d +  &  Conv2d +  & Conv2d +  &\\
     \bf Layer & LeakyReLu &  BatchNorm2d +  & BatchNorm2d +  & Conv2d\\
      & &  LeakyRelu & LeakyRelu & \\
     \midrule
     Input dim         & $28\times 28$ & $14\times 14$ & $7\times 7$ & $3\times 3$\\
     Nb input channels & 1 &  256 & 512 & 1024\\
     Kernel size       & 4 &  4   & 3 & 3\\
     Stride            & 2 &  2   & 2 & 1\\
     Padding           & 1 &  1   & 0 & 0\\
     Output dim        & $14\times 14$ &  $7\times 7$ & $3\times 3$ & $1\times 1$\\
     Nb output channels& 256 &  512 & 1024 & 1\\
     \bottomrule
 \end{tabular}
\vspace{25pt}
  \end{subtable}
  \begin{subtable}[b]{\textwidth}
  \centering
  \caption{Generator Network Architecture used in our experiments on MNIST data set. }
\vspace{10pt}
 \begin{tabular}{l|c|c|c|c}
    \toprule

      & ConvTranspose2d   &  ConvTranspose2d  & ConvTranspose2d  &\\
     \bf Layer & +BatchNorm2d &  +BatchNorm2d  & +BatchNorm2d  & ConvTranspose2d\\
      & +ReLu &  +Relu & +Relu & \\
     \midrule
     Input dim         & $1\times 1$ & $3\times 3$ & $7\times 7$ & $3\times 3$\\
     Nb input channels & 100 &  1024 & 512 & 256\\
     Kernel size       & 3 &  3   & 4 & 4\\
     Stride            & 1 &  2   & 2 & 2\\
     Padding           & 0 &  0   & 1 & 1\\
     Output dim        & $3\times 3$ &  $7\times 7$ & $14\times 14$ & $28\times 28$\\
     Nb output channels& 1024 &  512 & 256 & 1\\
     \bottomrule
 \end{tabular}
\vspace{25pt}
 \end{subtable}
 \begin{subtable}[b]{\textwidth}
      \centering
  \caption{Inverse Generator Network Architecture used in our experiments on MNIST data set. }
\vspace{10pt}
 \begin{tabular}{l|c|c|c|c|c}
    \toprule

      & Conv2d +  &  Conv2d +  & Conv2d +  & &\\
     \bf Layer & LeakyReLu &  BatchNorm2d +  & BatchNorm2d +  & Conv2d & Linear\\
      & &  LeakyRelu & LeakyRelu & &\\
     \midrule
     Input dim         & $28\times 28$ & $14\times 14$ & $7\times 7$ & $3\times 3$ & 100\\
     Nb input channels & 1 &  256 & 512 & 1024 & 1\\
     Kernel size       & 4 &  4   & 3 & 3 & -\\
     Stride            & 2 &  2   & 2 & 1 & -\\
     Padding           & 1 &  1   & 0 & 0 & -\\
     Output dim        & $14\times 14$ &  $7\times 7$ & $3\times 3$ & $1\times 1$ & 100\\
     Nb output channels& 256 &  512 & 1024 & 100 & 1\\
     \bottomrule
 \end{tabular}
 \end{subtable}
 \vspace{10pt}
\caption{Neural network architectures for the generator $g$, the critic $f$, and the
 lest inverse $h$ used in the experiments conducted on MNIST dataset. The negative slope parameter of the LeakyReLU is set to $0.2$.}
 \label{table:MNIST_architecture}
\end{table}

\end{document}